\documentclass[twoside, 11pt]{article}
\usepackage[utf8]{inputenc}
\usepackage[T1]{fontenc}
\usepackage[english]{babel}
\usepackage{amsmath, amssymb, amsthm, amstext, amsfonts, a4}
\usepackage{mathrsfs}
\usepackage{dsfont}
\usepackage{mathtools}
\usepackage[svgnames]{xcolor}
\usepackage{graphics,graphicx}
\usepackage{caption}
\usepackage{subcaption}
\usepackage{enumitem}
\usepackage{comment}
\usepackage{tgpagella}
\usepackage[affil-it]{authblk}
\usepackage{empheq}

\usepackage{hyperref}
\hypersetup{%
colorlinks=true,
linkcolor=Navy,
citecolor=Brown,
urlcolor=Navy
}

\usepackage{geometry}
\numberwithin{equation}{section}
\theoremstyle{plain}
	\newtheorem{theorem}{Theorem}[section] 
	\newtheorem{proposition}[theorem]{Proposition}       
	\newtheorem{lemma}[theorem]{Lemma}
	\newtheorem{corollary}[theorem]{Corollary}
        
\theoremstyle{definition}
	\newtheorem{definition}[theorem]{Definition}

\theoremstyle{remark}

\renewenvironment{proof}{\smallskip\noindent\emph{\textbf{Proof.}}%
  \hspace{1pt}}{\hspace{-5pt}{\nobreak\quad\nobreak\hfill\nobreak%
    $\square$\vspace{2pt}\par}\smallskip\goodbreak}

\newenvironment{proofof}[1]{\smallskip\noindent\emph{\textbf{Proof~of~#1.}}%
  \hspace{1pt}}{\hspace{-5pt}{\nobreak\quad\nobreak\hfill\nobreak%
    $\square$\vspace{2pt}\par}\smallskip\goodbreak}
\newcommand{\ds}[1]{\displaystyle{#1}}
\newcommand{\limit}[2]{{\ \underset{#1 \to #2}{\longrightarrow} \ }}
\newcommand{\1}{\mathbf{1}} 
\renewcommand{\d}[1]{\mathinner{\mathrm{d}{#1}}} 
\newcommand{\p}{\partial} 
\newcommand{\eps}{\mathrm{\varepsilon}}
\newcommand{\sgn}{\mathop{\rm sgn}}
\newcommand{\abs}[1]{{\left|#1\right|}}
\newcommand{\open}[2]{\mathopen]#1, #2 \mathclose[}
\newcommand{\norm}[1]{{\left\|#1\right\|}}

\newcommand{\N}{\mathbb{N}} 
\newcommand{\Z}{\mathbb{Z}} 
\newcommand{\R}{\mathbb{R}} 

\newcommand{\Czero}{\mathbf{C}} 
\newcommand{\Ck}[1]{\mathbf{C}^{#1}} 
\newcommand{\Cc}[1]{\mathbf{C}_\mathbf{c}^{#1}}
\newcommand{\Lip}{\mathbf{Lip}} 
 
\renewcommand{\L}[1]{\mathbf{L}^{#1}} 
\newcommand{\Lloc}[1]{\mathbf{L}_{\mathbf{loc}}^{#1}} 
\newcommand{\W}[2]{\mathbf{W}^{#1, #2}} 
\newcommand{\Wloc}[2]{\mathbf{W}_{\mathbf{loc}}^{#1, #2}}
 
\newcommand{\Hloc}[1]{\mathbf{H}_{\mathbf{loc}}^{#1}}
\newcommand{\BV}{\mathbf{BV}} 
\newcommand{\SBV}{\mathbf{SBV}} 
\newcommand{\TV}{\mathbf{TV}}

\newcommand{\god}{\mathbf{G}}
\newcommand{\fint}{F_{\mathrm{int}}}

\newcommand{\cG}{\mathcal{G}}
\newcommand{\cR}{\mathcal{R}}
\newcommand{\cU}{\mathcal{U}}
\newcommand{\sV}{\mathscr{V}}
\newcommand{\cM}{\mathcal{M}}

\begin{document}

\title{\textbf{Scalar Approach to ARZ-Type Systems of Conservation Laws}}

\author{Abraham Sylla$^1$}

\date{}

\maketitle

\footnotetext[1]{\texttt{abraham.sylla@u-picardie.fr} \\
LAMFA CNRS UMR 7352, Université de Picardie Jules Verne \\
33 rue Saint-Leu, 80039 Amiens (France) \\
ORCID number: 0000-0003-1784-4878}

\thispagestyle{empty}

\begin{abstract}
    We are interested in 2 $\times$ 2 systems of conservation laws of special structure, including generalized Aw-Rascle and Zhang (GARZ) models for road traffic. The simplest representative is the Keyfitz-Kranzer system, where one equation is nonlinear and not coupled to the other, and the second equation is a linear transport equation whose coefficients depend on the solution of the first equation. In GARZ systems, the coupling is stronger, they do not have the “triangular” structure of Keyfitz-Kranzer.

    In our setting, we claim that it makes sense to address these systems \textit{via} a kind of splitting approach. Indeed, [E. Y. Panov, Instability in models connected with fluid flows II, 2008] proposes a robust framework for solving linear transport equations with divergence free coefficients. Our idea is to use this theory for the second equation of GARZ systems, and to exploit discontinuous flux theory advances for the first equation of the system. 
\end{abstract}

\textbf{2020 AMS Subject Classification:} 65M08, 65M12, 35L65, 76A30.

\textbf{Keywords:} System of Conservation Laws, Discontinuous Flux, Finite Volume Scheme



\section{Introduction}

This paper deals with the theoretical and numerical analysis of a 2 $\times$ 2 system of coupled conservation laws, whose aim is to model self-organized road traffic. The first equation expresses the conservation of mass. The second equation, whose unknown $w$ is the state of orderliness in drivers' behavior, is a transport equation with discontinuous coefficients. The resulting system is a variant of generalized second-order models (GSOM) of traffic inspired by the now-classical Aw-Rascle and Zhang model, see \cite{AR2000, Zhang2002}:
\begin{empheq}[left=\empheqlbrace]{align}
    \label{eq:gsom}
    \p_t \rho + \p_x \left(f(\rho, w) \right) & = 0 \nonumber \\
    \p_t \left(\rho w \right) + \p_x \left(f(\rho, w) w\right) & 
    = 0.
\end{empheq}

The flux function takes the form 
\begin{equation}
    \label{eq:Flux}
    f(\rho, w) \coloneq \rho \sV(\rho, w),
\end{equation}

where for some velocity functions $V_{\min}, V_{\max} \in \W{2}{\infty}(\open{0}{1}, \R^+)$, 
\begin{equation}
    \label{eq:Velocity}
    \sV(\rho, w) \coloneq (1-w) V_{\min}(\rho) + w V_{\max}(\rho).
\end{equation}

Above, $V_{\min}, V_{\max}$ are the two levels of traffic velocity; one for the ordered regime of traffic and the other for the disordered regime. 
The actual velocity $\sV$ is a convex combination of the two regimes' velocities with $w(t,x) \in [0, 1]$ representing the state of orderliness of the traffic at time $t$ and position $x$. The system \eqref{eq:gsom} appears as the local version of the one studied by the authors in \cite{GSOM2022}, see the references within for a discussion about macroscopic traffic models of the same type.

On the mathematical side, the idea is to apply scalar equations techniques to study the system \eqref{eq:gsom}, an idea reminiscent of the works \cite{KK1990,KR2013,KPV2014, AAV2016,AMSV2020, AVV2021}. The first equation is treated as a conservation laws with spatial discontinuous flux. For piecewise regular (in time/space) discontinuities, the theory has more or less reached a level of completeness, see for instance \cite{AG2003, AJG2004,AGM2005,AKR2011, Diehl1995,GR1992,KRT2003,KT2004,Towers2000}. We deal with the second equation using the theory of renormalized solutions of linear transport equations with discontinuous coefficients put forward in \cite{Panov2008}. A property of this equation is that the piecewise constant structure of an initial datum is preserved for all time. This is the main reason why the initial datum $w_o$ is chosen piecewise constant in Definition \ref{def:gsom}.

All proofs are deferred to Section \ref{sec:Proofs}.


\section{Main result}
\label{sec:Main}

Throughout, we assume that the velocity functions satisfy
\begin{equation}
	\label{eq:VelocitiesAssumption}
	V_{\min}, \; V_{\max} \text{ are non-increasing, } \; V_{\min} \leq V_{\max} \; 
    \text{ and } \; V_{\min}(1) = V_{\max}(1) = 0,
\end{equation} 

and that the flux defined in \eqref{eq:Flux} is genuinely nonlinear in the sense that 
\begin{equation}
    \label{eq:StrictConcavity}
    \exists \eta > 0, \; \forall (\rho, w) \in [0, 1]^2, \quad \p_{\rho}^2 f(\rho, w) <-\eta.
\end{equation}

Assumption \eqref{eq:StrictConcavity} ensures that for all $w\in [0, 1]$, $\rho\mapsto f(\rho, w)$ is strictly concave, which is a usual requirement in the context of traffic flow modelling. 

The following definition is a combination of the theory of discontinuous flux for the 
first equation and renormalized solutions of transport equations with discontinuous 
coefficients for the second one. 

As always with nonlinear conservation laws, the Kruzhkov entropy flux \cite{Kruzhkov1970} is used:
\begin{equation}
    \label{eq:EntropyFlux}
    \forall \rho, w, k \in [0, 1], \quad
    \Phi(\rho, w, k) \coloneq \sgn(\rho - k) (f(\rho, w) - f(k, w)).
\end{equation}

\begin{definition}
    \label{def:gsom}
    Let $\rho_o, w_o \in \L{\infty}(\R, [0, 1])$. Assume that $w_o$ is piecewise constant: there exist $M \in \N \setminus \{0\}$ ordered real numbers $(d_m)_{m \in [\![1; M]\!]}$ and values $(a_m)_{m \in [\![0; M]\!]} \subset [0, 1]$ such that for all $x \in \R$,
    \begin{equation}
        \label{eq:InitialOrg}
        w_o(x) =
        \left\{ 
        \begin{array}{cl}
            a_0 & \text{if} \; x < d_1 \\
            a_{m} & \text{if} \; x \in [d_m, d_{m+1} \mathclose[, \; 
            m \in [\![1; M-1]\!] \\ 
            a_{M} & \text{if} \; x \geq d_M. 
        \end{array}
        \right.
    \end{equation}
    
    We say that $(\rho, w) \in \L{\infty}(\R^+ \times \R, [0, 1])^2$ is an entropy solution to \eqref{eq:gsom}-\eqref{eq:Flux} with initial datum $(\rho_o, w_o)$ if the following points hold.

    \begin{enumerate}[label=\bf(d\arabic*)]
        \item \label{item:regularity} 
        There exist $M$ ordered curves $y_m \in\Lip(\open{t_m}{T_m}, \R)$ such that if $\Gamma_m \coloneq \{(t, y_m(t)),\;t \in [t_m, T_m\mathclose[\}$, then $w$ is constant in each connected component of $(\R^+ \times \R) \setminus \cup_{m=1}^M \Gamma_m$.

        \item \label{item:EntropyInequalities} For all test functions $\varphi \in \Cc{\infty}(\R^+ \times \R, \R^+)$ and for all $k \in [0, 1]$, we have:
        \begin{empheq}{align}
            \label{eq:EI}
            \int_0^{+\infty} \int_\R \bigl( \abs{\rho - k} \p_t \varphi 
            + \Phi(\rho, w, k) \p_x \varphi \bigr) \d{x} \d{t} 
            + \int_\R \abs{\rho_o(x) - k} \varphi(0, x) \d{x} \nonumber \\
            + \sum_{m=1}^M \int_{t_m}^{T_m}
            \abs{f(k, w(t, y_m(t)+)) - f(k, w(t, y_m(t)-))} \varphi(t, y_m(t)) \d{t} \geq 0.
        \end{empheq}

        \item \label{item:WeakFormulation} For all test functions $\phi \in \Cc{\infty}(\R^+ \times \R, \R)$, it holds:
        \begin{equation}
            \label{eq:WF}
            \int_0^{+\infty} \int_\R \bigl( \rho w \p_t \phi + f(\rho, w) w \p_x \phi 
            \bigr) \d{x} \d{t} + \int_\R \rho_o(x) w_o(x) \phi(0, x) \d{x} = 0.
        \end{equation}
    \end{enumerate}
\end{definition}

The main drawback of Definition \ref{def:gsom} is the assumption that the initial orderliness marker $w_o$ is piecewise constant, and can therefore seem artificial. The difficulty is that the discontinuities in the first PDE of \eqref{eq:gsom} are both time and space dependent. In the stationary case ($w = w(x)$), a more natural way to treat the discontinuities is to replace \eqref{eq:EI} by the so-called adapted entropy inequalities, see \cite{AP2005, Panov2009}. The advantage of \cite[Definition 3.1]{AP2005} is that the interfaces, \textit{i.e.} the discontinuity curves of flux need not any special treatment. The setback is that the definition relies on a family of stationary solutions to the first PDE of \eqref{eq:gsom}, which is not available if $w$ also depends on time.

Finally, let us stress that even though for a fixed piecewise constant $w$, inequalities \eqref{eq:EI} ensure uniqueness for the first PDE in \eqref{eq:gsom}, Definition \ref{def:gsom} does not necessarily single out the physically relevant weak solution, unlike the classical theory of entropy solution for systems.

We build a numerical scheme and prove that the limit is an entropy solution in the sense of Definition \ref{def:gsom}. In \cite{GSOM2022}, the authors analyzed the $\BV$ stability to obtain compactness for the approximate density. This approach cannot be adapted here since in the context of conservation laws with discontinuous flux, no $\BV$ bounds are expected. For that reason, we rely here on compensated compactness method to obtain compactness for the approximate density. On the other hand, the approximation of the transport equation for $w$ in \eqref{eq:gsom} is still 
obtained exploiting the idea of propagation along characteristics. 

In the definition of our scheme, we require that the density is away from the vacuum at all times, see \eqref{eq:NumericalCharacteristics}. To ensure it, we reinforce the assumptions on the initial datum and on the velocities:
\begin{equation}
    \label{eq:vacuum}
    \exists \eps \in \open{0}{1}, \; \forall \rho \in [0, \eps], 
    \quad V_{\min}(\rho) = V_{\max}(\rho).
\end{equation}

In the context of traffic modelling, it makes sense to assume that the phase transition 
process only takes place at moderate and high densities. 
With Assumption \eqref{eq:vacuum}, the model \eqref{eq:gsom}-\eqref{eq:Flux} presents a 
two-phase behavior with $\rho \in [0, \eps]$ corresponding to the free traffic flow 
phase while $\rho > \eps$ corresponds to the congested traffic.
The situation we have in mind is depicted in Figure \ref{fig:PhaseTransition}, where we choose
\[
    V_{\min}(\rho) \coloneq 1 - \rho, \quad 
    V_{\max}(\rho) = V_{\min}(\rho) + A (1 - \rho) ((\rho - \eps)^{+})^{2}, \; A \in \R^+.
\]

$V_{\max}$ is nonincreasing when $A  \leq \frac{3}{(1-2\eps)^2}$ and $\eps \leq \frac{1}{3}$, and $\rho \mapsto \rho V_{\max}(\rho)$ is concave if $A \leq \frac{8}{4\eps^2 -4\eps+3}$.

\begin{figure}[!htp]
    \begin{center}
    \includegraphics[scale= 0.38]{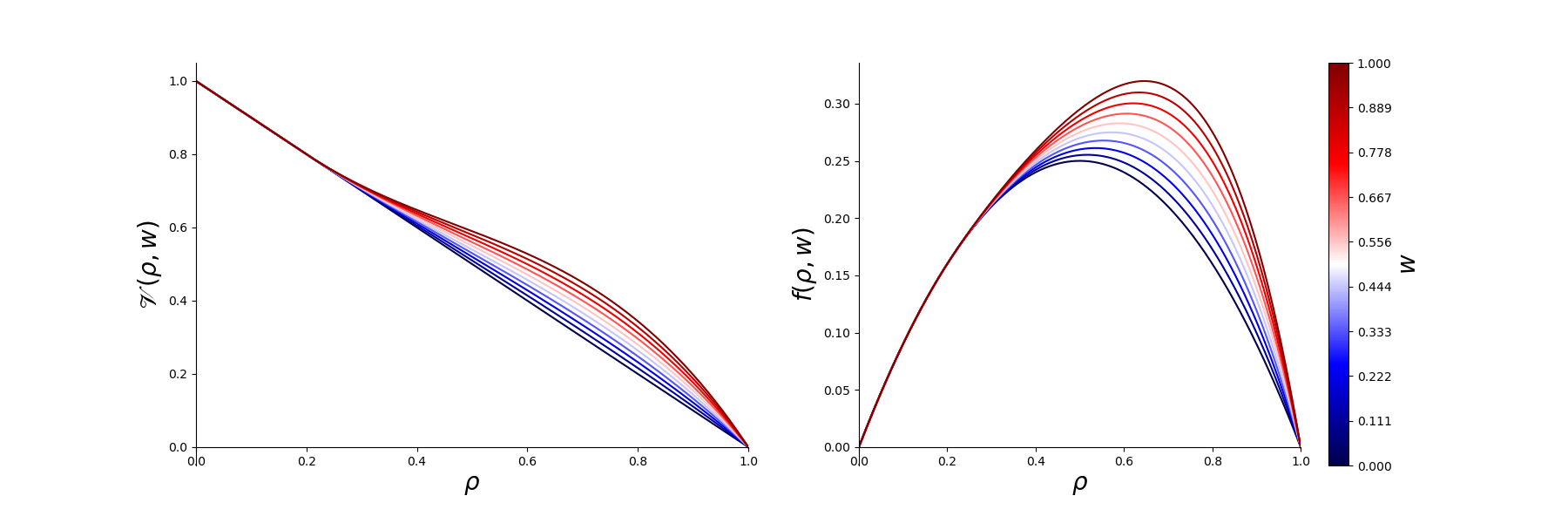}
    \end{center}
    \vspace{-5mm}
    \caption{Illustration of the phase transition.}
    \label{fig:PhaseTransition}
\end{figure}

\newpage

Let us describe the scheme. 
Fix a spatial-temporal mesh size $(\Delta x, \Delta t)$. With reference to 
\eqref{eq:vacuum}, we assume that the ratio $\lambda \coloneq \Delta t/\Delta x$ satisfies the CFL condition 
\begin{equation}
    \label{eq:CFL}
    \lambda \max\left\{2, \frac{1}{\eps}\right\} L \leq 1, \quad 
    L\coloneq \sup_{p, \omega \in [0, 1]} \abs{\p_\rho f(p, \omega)}.
\end{equation}

Notice that as $\eps \to 0+$, the CFL condition blows up.

For all $n \in \N$ and $j \in \Z$, set the notations 
\[
    t^n \coloneq n \Delta t, \quad x_j \coloneq j\Delta x, \; 
    x_{j+1/2} \coloneq (j+1/2)\Delta x, \; 
    I_j \coloneq \open{x_{j-1/2}}{x_{j+1/2}}.
\]

Discretize the initial datum:
\[
    \forall j \in \Z, \quad  
    (\rho_j^0, w_j^0) \coloneq \frac{1}{\Delta x} \int_{I_j} (\rho_o(x), w_o(x)) \d{x}.
\]

Fix $n \in \N$. Let us explain how, given $(\rho_j^n, w_j^n)_{j \in \Z}$, we determine  
$(\rho_j^{n+1}, w_j^{n+1})_{j \in \Z}$. 

We start by updating $(\rho_j^n)_{j \in \Z}$. Let $\rho^n$ and $\omega^n$ be defined as 
\[
    (\rho^n, \omega^n) \coloneq \sum_{j \in \Z} (\rho_j^n, w_j^n) \1_{I_j}.
\]

We call $\cU^n = \cU^n(t, x)$ the unique entropy solution in the sense of 
Theorem \ref{th:EquivalentDefinitions} to the following Cauchy problem, set on 
$\open{t^n}{t^{n+1}} \times \R$:
\begin{empheq}[left=\empheqlbrace]{align}
    \label{eq:schemeStep1}
    \p_t \cU(t, x) + \p_x \left(f(\cU(t, x), \omega^n(x)) \right) & = 0 \nonumber \\
    \cU(t^n, x) & = \rho^n(x).
\end{empheq}

Define $\rho_j^{n+1}$ as 
\begin{equation}
    \label{eq:MF1}
    \rho_j^{n+1} \coloneq \frac{1}{\Delta x} \int_{I_j} \cU^n(t^{n+1}-, x) \d{x}.
\end{equation}

Let us specify that under the CFL condition, waves emanating from the interfaces 
$x = x_{j+1/2}$ do not interact. Consequently, \eqref{eq:MF1} rewrites as
\begin{equation}
    \label{eq:MF2}
    \rho_j^{n+1} = \rho_j^n - \lambda (F_{j+1/2}^n (\rho_j^n, \rho_{j+1}^n) 
    - F_{j-1/2}^n (\rho_{j-1}^n, \rho_j^n)),
\end{equation}

where following \cite{AJG2004}, for all $(u, v) \in [0, 1]^2$,
\begin{equation}
    \label{eq:InterfaceFluxFirst}
    F_{j+1/2}^n (u, v) \coloneq 
    \min \biggl\{\god_j^n (u, \eps), \; \god_{j+1}^n(1, v) \biggr\}
\end{equation}

is the flux across the interface $x = x_{j+1/2}$, with $\god_j^n$ 
denoting the Godunov numerical flux associated with $f_j^n \coloneq \rho \mapsto f(\rho, w_j^n)$. Next, we update $(w_j^n)_j$ having in mind the non-conservative form of the second equation in \eqref{eq:gsom}. More precisely, with reference to \eqref{eq:Velocity}, we essentially discretize 
\[
    \p_t w + \sV(\rho, w) \p_x w = 0 
\]

following two steps.
\begin{itemize}
    \item \textbf{Step 1:} Drawing the numerical characteristics:
    \begin{equation}
        \label{eq:NumericalCharacteristics}
        \forall j \in \Z, \quad 
        X_{j-1/2}^{n+1} \coloneq x_{j-1/2} + s_{j-1/2}^{n} \; \Delta t, \quad 
        s_{j-1/2}^{n} \coloneq \frac{F_{j-1/2}^n (\rho_{j-1}^n, \rho_j^n)}{\rho_j^{n+1}}.
    \end{equation}
    
    \item \textbf{Step 2:} Averaging step:
    \begin{equation}
        \label{eq:MFw}
        \forall j \in \Z, \quad w_{j}^{n+1} \coloneq \left(1 - \lambda s_{j-1/2}^{n} \right) 
        w_j^n  + \lambda s_{j-1/2}^{n} w_{j-1}^n.
    \end{equation}
\end{itemize} 

Notice that one will need to prove that \eqref{eq:NumericalCharacteristics} always has a sense, meaning that $\rho_j^{n+1}$ never vanishes. It is precisely the role of Assumption \eqref{eq:vacuum} which, combined with the CFL condition and the fact that $\rho_o\in [\eps, 1]$, will ensure that $\rho_j^{n} \in [\eps, 1]$ for all $n \in \N$ and $j \in \Z$.
\bigskip 

For the numerical analysis, we define the approximate solutions 
$(\rho_\Delta, w_\Delta)$ the following way:
\[
    \rho_\Delta \coloneq \sum_{n=0}^{+\infty} \cU^n 
    \1_{[t^n , t^{n+1} \mathclose[}, \quad 
    w_\Delta \coloneq \sum_{n=0}^{+\infty} \sum_{j \in \Z} w_j^n 
    \1_{[t^n , t^{n+1} \mathclose[ \times I_j}. 
\]

We can now state our main result.

\begin{theorem}
    \label{th:ConvergenceScheme}
    Assume that $V_{\min}, V_{\max} \in \W{2}{\infty}(\open{0}{1}, \R^+)$ satisfy 
    \eqref{eq:VelocitiesAssumption}-\eqref{eq:StrictConcavity}-\eqref{eq:vacuum}.
    Fix $\rho_o \in \L{\infty}(\R, [\eps, 1])$ and $w_o \in \BV(\R, [0, 1])$. 
    Then, as $(\Delta x, \Delta t) \to (0, 0)$ while satisfying \eqref{eq:CFL}, along a subsequence, $(\rho_\Delta, w_\Delta)$ converges to a weak solution $(\rho, w)$ of \eqref{eq:gsom}-\eqref{eq:Flux}.

    If moreover, $w_o$ is piecewise constant, then $(\rho, w)$ is an entropy solution in the sense of Definition \ref{def:gsom}. 
\end{theorem}

As an immediate consequence of Theorem \ref{th:ConvergenceScheme}, we obtain an existence result.

\begin{theorem}
    \label{th:ExistenceScheme}
    Assume that $V_{\min}, V_{\max} \in \W{2}{\infty}(\open{0}{1}, \R^+)$ satisfy 
    \eqref{eq:VelocitiesAssumption}-\eqref{eq:StrictConcavity}-\eqref{eq:vacuum}. Fix $\rho_o \in \L{\infty}(\R, [\eps, 1])$ and $w_o \in \BV(\R, [0, 1])$. Then, there exists a weak solution $(\rho, w)$ to \eqref{eq:gsom}-\eqref{eq:Flux} with 
    initial datum $(\rho_o, w_o)$. 
    
    If moreover, $w_o$ is piecewise constant, then $(\rho, w)$ is an entropy solution in the sense of Definition \ref{def:gsom}.  
\end{theorem}

The optimal assumption on $w_o$ is probably that $w_o$ is piecewise $\SBV(\R)$, provided that one can prove that the piecewise $\SBV(\R)$ regularity is propagated by the transport equation, which is immediate in the case of piecewise constant initial $w_o$, see Theorem \ref{th:BeyondPanov}. Allowing piecewise $\SBV(\R)$ $w_o$ would also permit to add a source term in the second equation of \eqref{eq:gsom}. That way, a dynamic for the state of orderliness $w$ could be incorporated.

\newpage

\section{Numerical integrations}

This section contains numerical integrations of \eqref{eq:gsom}, based on the scheme 
described in Section \ref{sec:Main}. We choose $V_{\min}$ and $V_{\max}$ as described in Section \ref{sec:Main} with 
\[
    \eps = 0.2, \quad A \coloneq 0.9 \times \min \left\{ \frac{3}{(1-2\eps)^2}, \frac{8}{4\eps^2 -4\eps +3}\right\},
\]

so that $V_{\min}, V_{\max} \in \W{2}{\infty}(\open{0}{1}, \R^+)$, and \eqref{eq:VelocitiesAssumption}-\eqref{eq:StrictConcavity} hold.

First, consider Riemann-type initial data
\[
    (\rho_o(x), w_o(x)) = 
    \left\{
        \begin{aligned}
            (\rho_l, w_l) & \text{ if } \; x < 0 \\
            (\rho_r, w_r) & \text{ if } \; x \geq 0,
        \end{aligned}
    \right.
\]

for given $\rho_l, \rho_r \in [\eps, 1]$ and $w_l, w_r \in [0, 1]$. For the same $w_o$, we present the numerical solutions obtained with $(\rho_l, \rho_r) = (0.4, 0.9)$ (Figure \ref{fig:rp1}) and 
$(\rho_l, \rho_r) = (0.9, 0.4)$ (Figure \ref{fig:rp2}). As expected from Definition \ref{def:gsom} and Theorem \ref{th:ConvergenceScheme}, in both cases, at all times $t \in \R^+$, $w(t)$ is piecewise constant, taking the same values as $w_o$, separated by a curve $y$.

\begin{figure}[!htp]
    \begin{center}
    \includegraphics[scale=0.25]{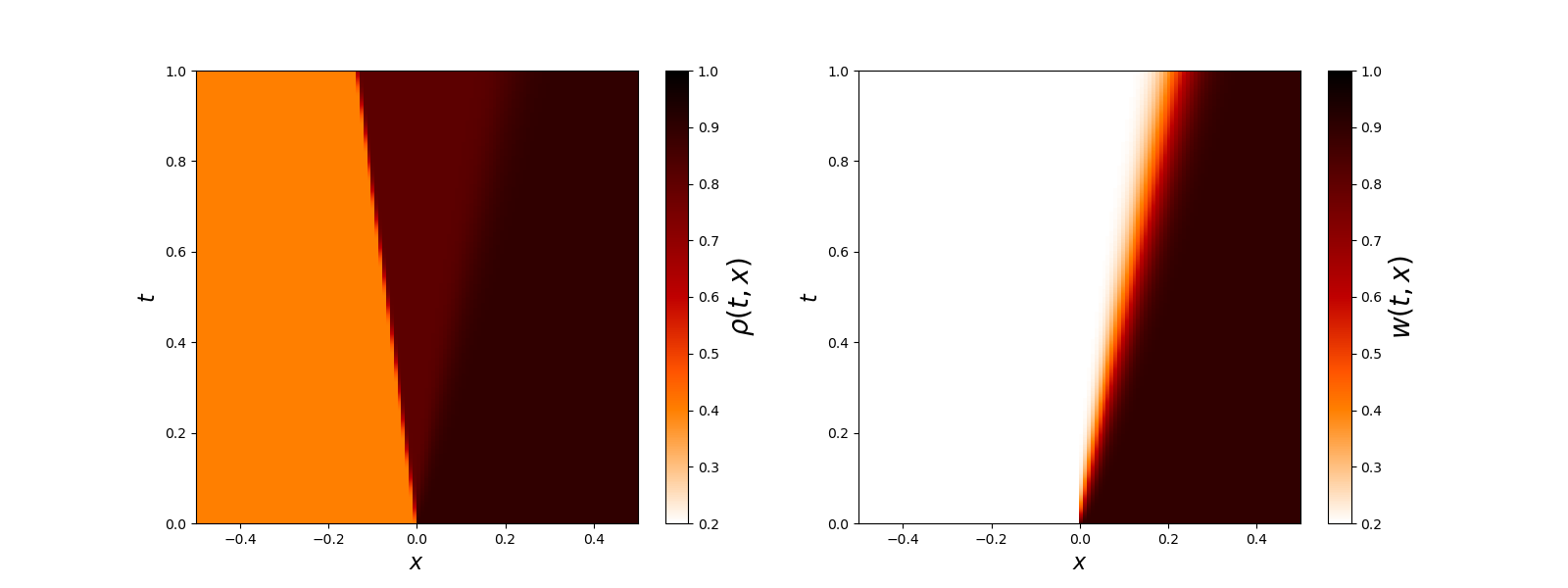}
    \end{center}
    \vspace{-5mm}
    \caption{Shock of the first family followed by a contact discontinuity.}
    \label{fig:rp1}
\end{figure}

\begin{figure}[!htp]
    \begin{center}
    \includegraphics[scale=0.25]{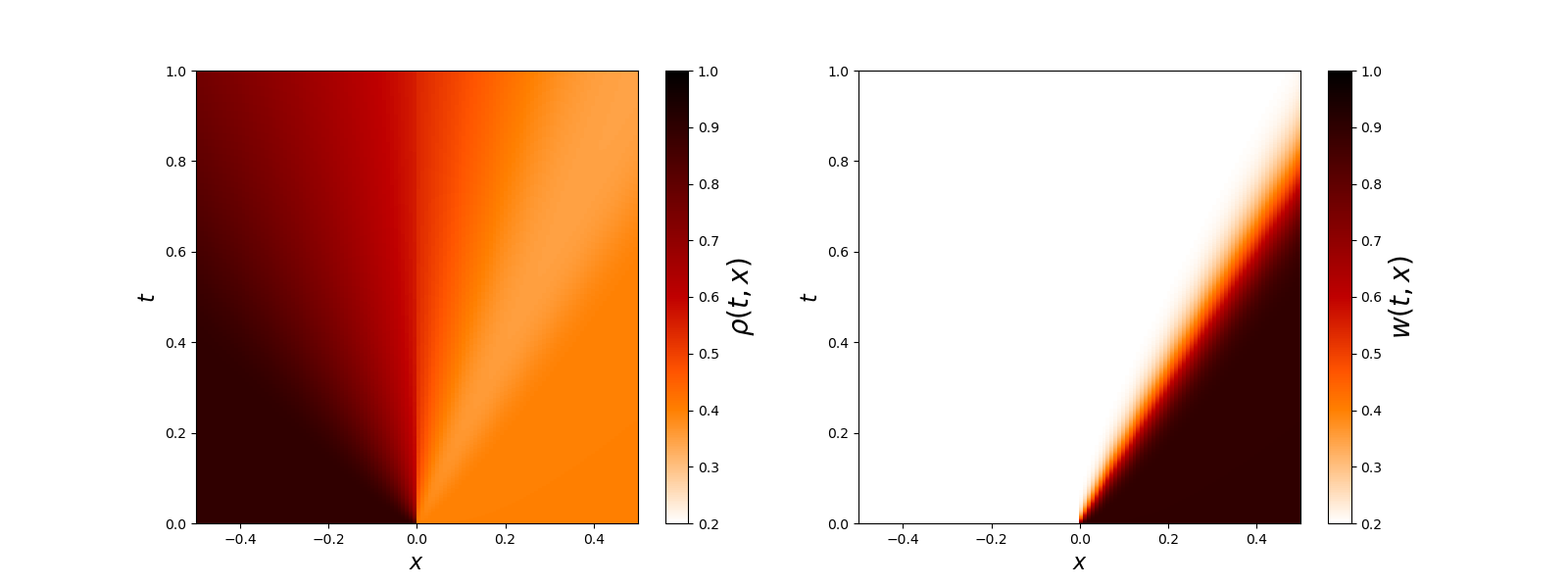}
    \end{center}
    \vspace{-5mm}
    \caption{Rarefaction of the first family followed by a contact discontinuity.}
    \label{fig:rp2}
\end{figure}

In Figure \ref{fig:rp1}, $\rho$ is composed of the succession of two discontinuities, the first one being an entropic shock, while in Figure \ref{fig:rp2}, the values of $\rho_o$ create a classical rarefaction wave followed by a contact discontinuity. In both cases, the speed of the last discontinuity in $\rho$ seems to coincide with the speed of the $w$-discontinuity. The numerical solution seems to coincide with the solution to the Riemann problem for \eqref{eq:gsom}, seen as an hyperbolic system, see for instance \cite{GS2018}. In Figure \ref{fig:rp1}, we see a Lax shock in $\rho$ followed by a contact discontinuity in $(\rho, w)$, and in Figure \ref{fig:rp2}, we see a rarefaction wave in $\rho$ followed by a contact discontinuity in $(\rho, w)$. Finally, notice the diffusion in the $w$ discontinuities, consequence of the averaging step \eqref{eq:MFw}.

\newpage

Now, we choose $w_o$ piecewise constant taking three values, and a constant initial density. The numerical solution is presented in Figure \ref{fig:pwc}.

\begin{figure}[!htp]
    \centering 
    \begin{subfigure}{0.4\textwidth}
        \centering
        \includegraphics[width=\textwidth]{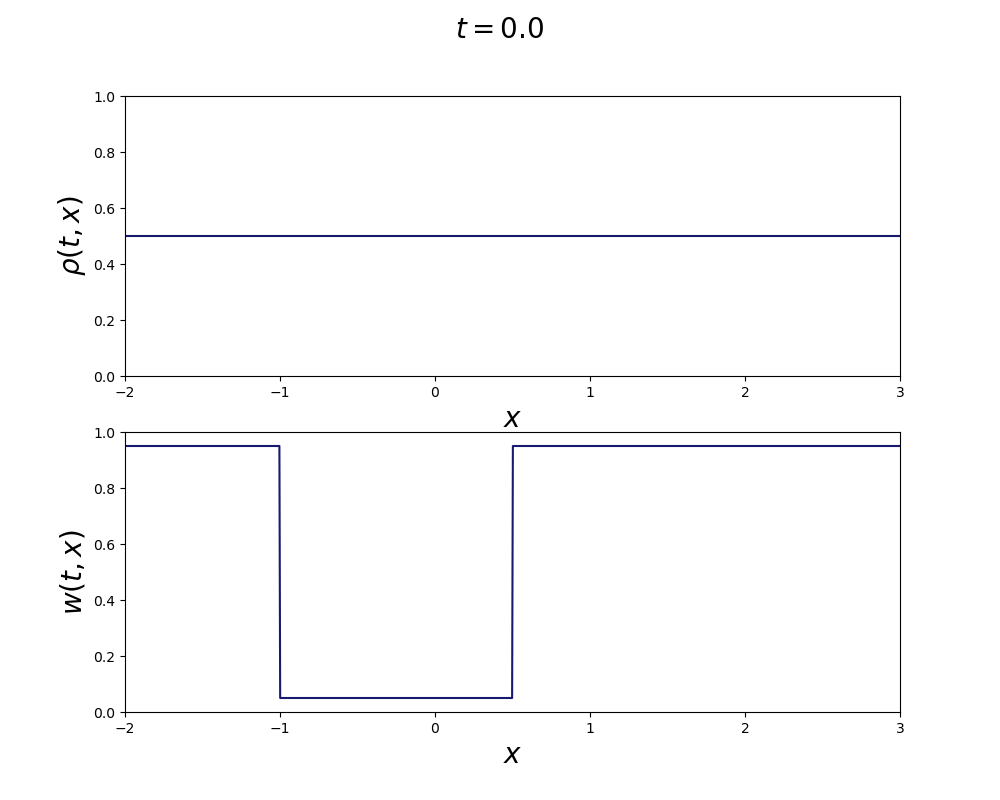}
    \end{subfigure}
    \hspace{0pt}
    \begin{subfigure}{0.4\textwidth}
        \centering
        \includegraphics[width=\textwidth]{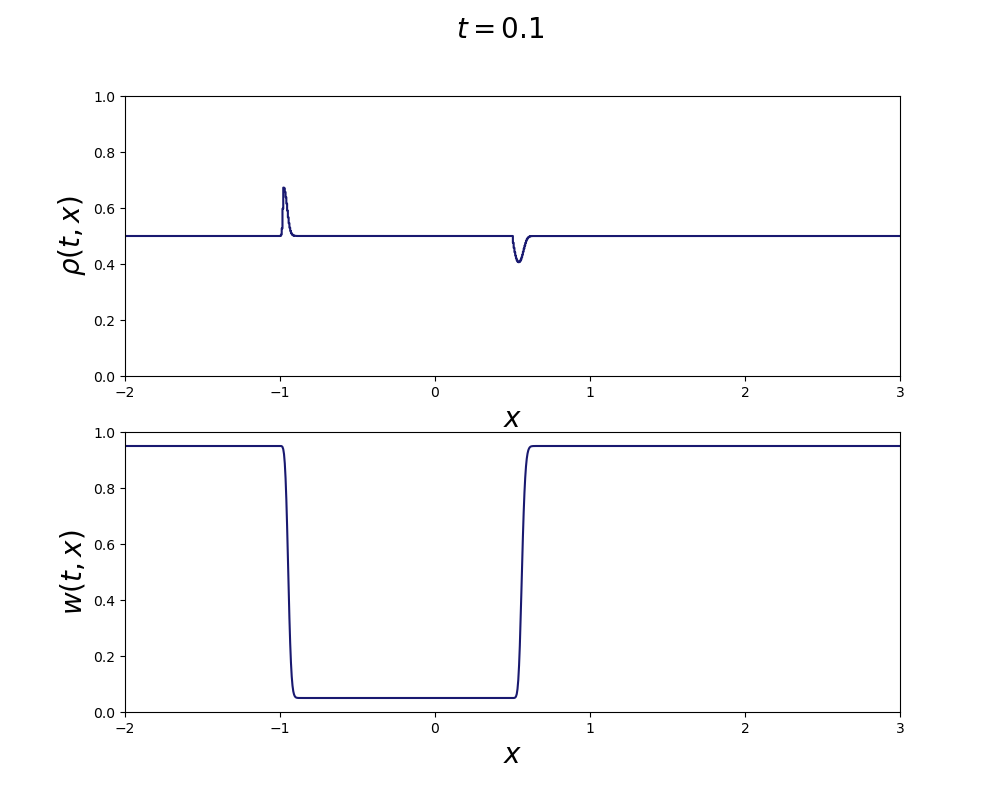}
    \end{subfigure}
    \begin{subfigure}{0.4\textwidth}
        \centering
        \includegraphics[width=\textwidth]{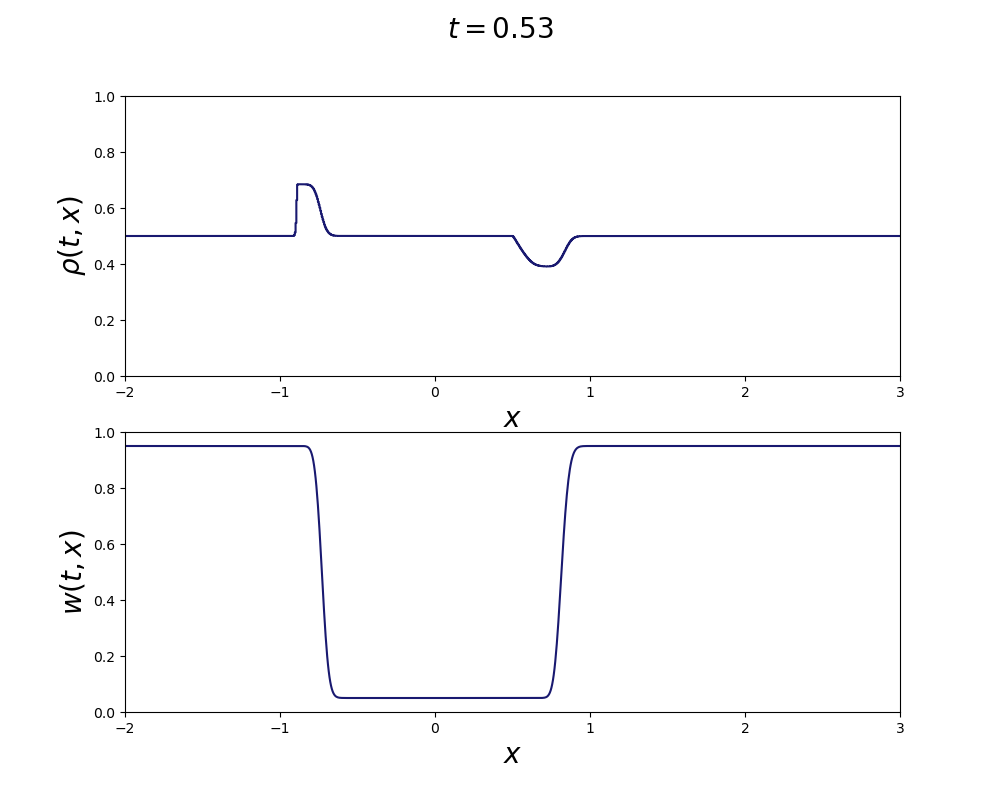}
    \end{subfigure}
    \hspace{0pt}
    \begin{subfigure}{0.4\textwidth}
        \centering
        \includegraphics[width=\textwidth]{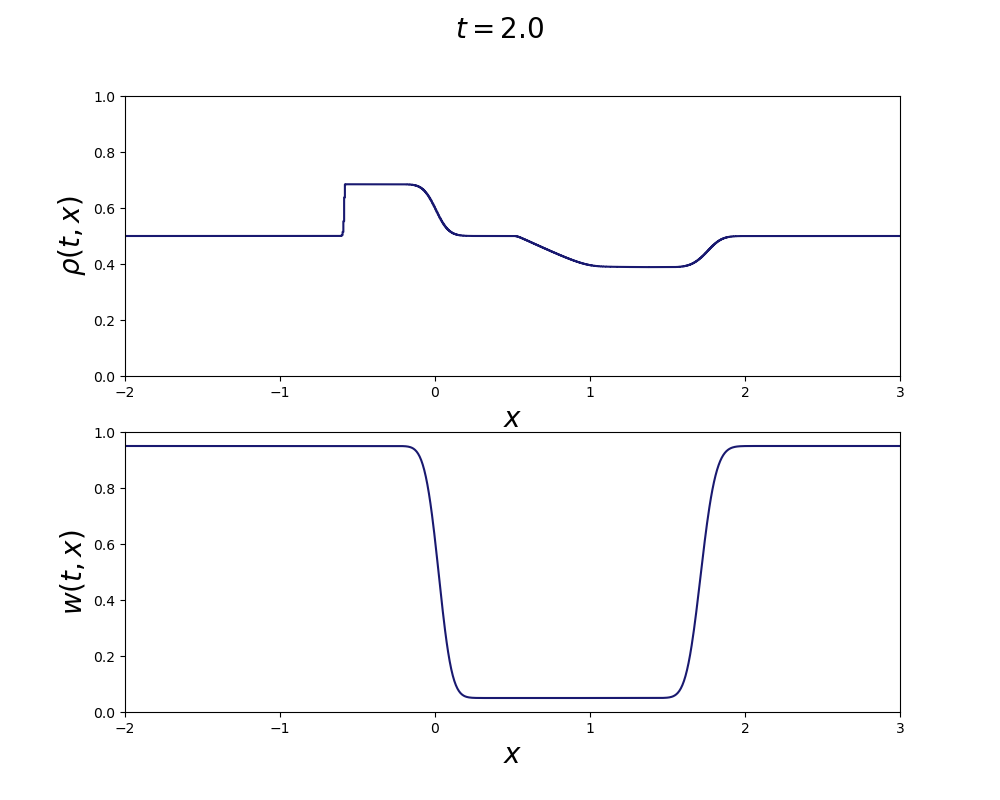}
    \end{subfigure}
    \caption{$w_o$ piecewise constant, taking three values.}
    \label{fig:pwc}
\end{figure}

In the $w$ variable, the solution has the expected structure. Indeed, for all $t \in \R^+$, $w(t)$ is piecewise constant, taking the same values as $w_o$, in the same order. Therefore, $w(t)$ has exactly the same total variation as $w_o$. At the points of discontinuity of $w_o$, new (and non classical) singularities appear due to the change of flux.

\section{Analytical Proofs}
\label{sec:Proofs}

\subsection{Preliminaries}
\label{ssec:BeyondPanov}

\begin{definition}
    Let $(\rho, V) \in \L{\infty}(\R^+ \times \R,\R^+)^2$ a divergence free vector field and $w_o \in \L{\infty}(\R, [0, 1])$.

    We say that $w \in \L{\infty}(\R^+ \times \R, \R)$ is a weak solution to 
    \begin{empheq}[left=\empheqlbrace]{align}
        \label{eq:PanovStructure}
        \p_t(\rho w) + \p_x (\rho V w) & = 0 \nonumber \\
        w(0, x) & = w_o(x)
    \end{empheq}
    
    if for all test functions $\phi \in \Cc{\infty}(\R^+ \times \R, \R)$, it holds:
    \[
        \int_{0}^{+\infty} \int_{\R} (\rho w) \p_t \phi + (\rho V w) \p_x \phi \; \d{x} \d{t} 
        + \int_\R \rho(0, x) w_o(x) \phi(0, x) \d{x} = 0.
    \]
\end{definition}

\begin{theorem}
    \label{th:BeyondPanov}
    Let $(\rho, V) \in \L{\infty}(\R^+ \times \R,\R^+)^2$ a divergence free vector field and $w_o \in \L{\infty}(\R, [0, 1])$.

    (i) There exists a unique weak solution $w$ to \eqref{eq:PanovStructure} with initial datum $w_o$.   

    (ii) If moreover, $w_o$ is piecewise constant, as described in Definition \ref{def:gsom}, then there exist $M$ ordered curves $y_m \in \Lip(\open{t_m}{T_m}, \R)$ such that setting $\Gamma_m \coloneq \{(t, y_m(t)),\;t \in [t_m, T_m\mathclose[\}$, $w$ is constant in each connected component of $(\R^+ \times \R) \setminus \cup_{m=1}^M \Gamma_m$.
\end{theorem}

The well-posedness stated in Theorem \ref{th:BeyondPanov} \textit{(i)} comes from \cite[Corollary A.8]{GSOM2022}. What we prove below is \textit{(ii)}, establishing that if $w_o$ is piecewise constant, then $w$ is piecewise constant as well.

\begin{proofof}{Theorem \ref{th:BeyondPanov} (ii)}
    For future references, set $F \coloneq \rho V$. Moreover, for simplicity, we prove the result in the case $M=1$. More precisely, let $w_l, w_r \in [0, 1]$ and define 
    \[
        w_o(x) \coloneq
        \left\{ 
        \begin{array}{cl}
            w_l & \text{if} \; x < 0 \\
            w_r & \text{if} \; x \geq 0. 
        \end{array}
        \right.
    \]

    \paragraph{Step 1: approximation.}
    
    Fix $\varphi\in\Cc{\infty}(\R,\R^+)$ a weight test function supported in $[-1,0]$. For all $k \in \N \setminus \{0\}$, consider the function 
    \[
        \psi_k (t, x) \coloneq k^2 \varphi(kt) \varphi(kx) \in \Cc{\infty}(\R^2, \R^+).
    \]

    We call $\rho_k \coloneq \rho * \psi_k + \frac{1}{k}$ and $F_k \coloneq F * \psi_k + \frac{1}{k}$ the  smooth approximations of the coefficients. Note that for all $k \in \N \setminus \{0\}$ and $(t, x) \in \open{0}{+\infty} \times \R$,
	\[
		\rho_k(t, x) 
		= \int_{\R} \int_{\R} \rho(t-s, x-y) \psi_k(s, y) \d{y} \d{s} + \frac{1}{k}
		= \int_{-1/k}^0 \int_{-1/k}^0 \rho(t-s, x-y) \psi_k(s, y) \d{y} \d{s} 
        + \frac{1}{k},
	\]

	and therefore the convolutions are well-defined on $\open{0}{+\infty} \times \R$. 
    
    The sequences $(\rho_k)_k$ and $(F_k)_k$ converge in 
    $\Lloc{1}(\open{0}{+\infty} \times \R, \R)$ to $\rho$ and $F$, respectively, and even if it means taking subsequences, we can assume that the convergence is almost everywhere on $\open{0}{+\infty} \times \R$. Note also that since $\rho \geq 0$, then $\rho_k \geq 1/k > 0$ for all $k \in \N \setminus \{0\}$. 
    
    It is readily checked that $\p_t \rho_k + \p_x F_k = 0$ in the sense of distributions; $\rho_k$ and $F_k$ being smooth, we deduce that the equality holds pointwise. Since $\rho_k$ does not vanish, the function $V_k \coloneq \frac{F_k}{\rho_k}$ is smooth. Moreover, we have the $\L{\infty}$ bound:
    \begin{equation}
        \label{eq:BeyondPanovProof}
        0 \leq V_k
		= \frac{(\rho V) * \psi_k + 1/k}{\rho * \psi_k + 1/k} 
        \leq \norm{V}_{\L{\infty}(\R^+ \times \R)} + 1.
    \end{equation}

    Using the method of characteristics, we can define $w_k \in \L{\infty}(\open{0}{+\infty} \times\R,\R)$ as the unique solution to
    \[
		\left\{
            \begin{aligned}
                \p_t w_k(t, x) + V_k(t, x) \; \p_x w_k(t, x) & = 0 \\
                w_k(0, x) & = w_o(x).
            \end{aligned}
        \right.
	\]

    More precisely, for all $(t, x) \in \open{0}{+\infty} \times \R$, solve the Cauchy problems  
    \[
        \left\{
            \begin{aligned}
                X_k'(s, t, x) & = V_k(s, X_k(s, t, x)) \\
                X_k(t, t, x) & = x
            \end{aligned}
        \right. \qquad \text{and} \qquad
        \left\{
            \begin{aligned}
                u_k'(s, t, x) & = 0 \\
                u_k(0, t, x) & = w_o(X_k(0, t, x)),
            \end{aligned}
        \right.
    \]

    and set $w_k(t, x) \coloneq u_k(t, t, x)$. Immediately, $w_k(t, x) = w_o(X_k(0, t, x)) \in \{w_l, w_r\}$ and more precisely,
    \[
        w_k(t, x) = 
        \left\{ 
        \begin{array}{cl}
            w_l & \text{if} \; x < \gamma_k(t) \\
            w_r & \text{if} \; x > \gamma_k(t),
        \end{array}
        \right.
    \]

    where $\gamma_k \in \Ck{1} \cap \Lip(\R^+, \R)$ is the nondecreasing curve solving 
    \[
        \left\{
            \begin{aligned}
                y'(s) & = V_k(s, y(s)) \\
                y(0) & = 0.
            \end{aligned}
        \right.
    \] 

    \paragraph{Step 2: compactness/convergence.} 

    From the estimate \eqref{eq:BeyondPanovProof}, we deduce that for all $T > 0$, $(\gamma_k)_k$ is bounded in $\W{1}{\infty}(\open{0}{T}, \R)$. Therefore, using a diagonal process, we can construct a subsequence and $y \in \Lip(\open{0}{+\infty}, \R)$ such that $(\gamma_k)_k$ converges to $y$ in $\Czero(\R^+, \R)$. Naturally, $y(0) = 0$, $y$ is nondecreasing and $\norm{y'}_{\L{\infty}(\R^+)} \leq \norm{V}_{\L{\infty}(\R^+ \times \R)} + 1$. Hence, $(w_k)_k$ converges a.e. on $\open{0}{+\infty} \times \R$ to $w$ defined by 
    \[
        w(t, x) = 
        \left\{ 
        \begin{array}{cl}
            w_l & \text{if} \; x < y(t) \\
            w_r & \text{if} \; x > y(t).
        \end{array}
        \right.
    \]

    Since for all $k \in \N \setminus \{0\}$, $w_k$ is a weak solution to  
    \[
		\left\{
            \begin{aligned}
                \p_t (\rho_k w_k) + \p_x (F_k w_k) & = 0 \\
                w_k(0, x) & = w_o(x),
            \end{aligned}
        \right.
	\]

    we deduce that $w$ is the unique weak solution to \eqref{eq:PanovStructure}. The proof is concluded.
\end{proofof}

The transport equation propagates the piecewise constant regularity. It is our intuition that the optimal assumption for $w_o$ is piecewise $\SBV(\R)$ regularity, but one would need to prove that this regularity is propagated by the transport equation.

\subsection{Discontinuous Flux}
\label{ssec:DiscontinuousFlux}

We recall and adapt for our use results regarding conservation laws with discontinuous flux. For the purpose of this section, let us fix $f_l, f_r \in \Lip([0, 1], \R^+)$ two strictly concave and bell-shaped fluxes. The complete computations can be found in \cite[Section 3.1]{Sylla2024}. 

To fix the ideas, and to highlight the link with \eqref{eq:Flux}, we assume that $f_l \geq f_r$. Consider:
\begin{equation}
    \label{eq:CLDiscontinuous}
    \p_t u + \p_x \left(F(x, u) \right) = 0, \quad
    F(x, u) \coloneq \left\{ 
        \begin{aligned}
            f_l(u) & \quad \text{if} \; x \leq 0 \\
            f_r(u) & \quad \text{if} \; x > 0.
        \end{aligned}
    \right.
\end{equation}

Since the works of \cite{AG2003, AJG2004,AGM2005,AKR2011, Diehl1995,GR1992,KRT2003,KT2004,Towers2000}, it is now well known in addition to conservation of mass (Rankine-Hugoniot condition)
\[
    \text{for a.e.} \; t > 0, \quad f_l(u(t, 0-)) = f_r(u(t, 0+)),
\]

an entropy criterion must be imposed at the interface to select one solution. This 
choice is often guided by physical considerations. Among the most common criteria, we 
can cite: the minimal jump condition \cite{GR1992, KMRWellBalanced,KMRTriangular}, the 
vanishing viscosity criterion \cite{BGS2019} or the flux maximization \cite{AGM2005}.

If $\Phi_l$, resp. $\Phi_r$, denotes the Kruzhkov entropy flux associated with 
$f_{l}$, resp. with $f_r$, then in this section, define $\Phi = \Phi(x, u)$ as the 
Kruzhkov entropy flux associated with $F$:
\[
    \forall x \in \R, \; \forall a, b \in \R, \quad 
    \Phi(x, a, b) \coloneq \sgn(a-b) (F(x, a) - F(x, b)) 
    = \left\{ 
        \begin{aligned}
            \Phi_l(a, b) & \quad \text{if} \; x \leq 0 \\
            \Phi_r(a, b) & \quad \text{if} \; x > 0.
        \end{aligned}
    \right.
\]

For the study of \eqref{eq:CLDiscontinuous}, we follow \cite{AKR2011}, where the traces 
of the solution at the interface $\{x=0\}$ are explicitly treated. It will be useful to 
have a name for the critical points of $f_{l,r}$, say $\alpha_{l, r}$.
Notice that for all $y \in \mathopen[0, \max f_r]$, the equation $f_l(u) = y$ admits 
exactly two solutions, $0 \leq S_l^-(y) < \alpha_l < S_l^+(y) \leq 1$. 
This motivates the following definition.

\begin{definition}
    \label{def:Germ}
    The admissibility germ for \eqref{eq:CLDiscontinuous} is the subset $\cG$ defined 
    as the union of the following subsets:
    \begin{equation}
        \label{eq:Germ}
        \begin{aligned}
            \cG_1 & \coloneq \left\{ (k_l, k_r) \in [0,1]^2 \; : \; k_r \leq \alpha_r, \; 
            k_l = S_l^-(f_r(k_r)) \right\} \\[5pt]
            \cG_2 & \coloneq \left\{ (k_l, k_r) \in [0,1]^2 \; : \; k_r \geq \alpha_r, \; 
            k_l = S_l^+(f_r(k_r)) \right\} \\[5pt]
            \cG_3 & \coloneq \left\{ (k_l, k_r) \in [0,1]^2 \; : \; k_r > \alpha_r, \; 
            k_l = S_l^-(f_r(k_r)) \right\}.
        \end{aligned}
    \end{equation} 
\end{definition}

The germ contains all the possible traces along $\{x=0\}$ of the solutions 
to \eqref{eq:CLDiscontinuous}. Notice that by construction, any couple in the germ 
satisfies the Rankine-Hugoniot condition. Conversely, some couples verifying the 
Rankine-Hugoniot condition have been excluded from the germ, more precisely the ones 
belonging to 
\[
    \left\{ (k_l, k_r) \in [0,1]^2 \; : \; k_r < \alpha_r, \; k_l = S_l^+ (f_r(k_r)) 
    \right\}. 
\]

The reason lies in the following proposition, see in particular \eqref{eq:MaximalGerm}.

\begin{proposition}
    \label{pp:GermL1DMax}
    Define $\cG$ as \eqref{eq:Germ}. Then, $\cG$ is a maximal $\L{1}$-dissipative germ, 
    meaning that the following points hold.

    (i) For all $(u_l, u_r) \in \cG$, $f_l(u_l) = f_r(u_r)$.

    (ii) Dissipative inequality:
    \begin{equation}
        \label{eq:DissipativeGerm}
        \forall (u_l, u_r), (k_l, k_r) \in \cG, \quad
        \Phi_l(u_l, k_l) - \Phi_r(u_r, k_r) \geq 0.
    \end{equation}
    
    (iii) Maximality condition: let $(u_l, u_r) \in [0, 1]^2$ such that 
    $f_l(u_l) = f_r(u_r)$. Then 
    \begin{equation}
        \label{eq:MaximalGerm}
        \forall (k_l, k_r) \in \cG, \; 
        \Phi_l(u_l, k_l) - \Phi_r(u_r, k_r) \geq 0 \implies (u_l, u_r) \in \cG.
    \end{equation}
\end{proposition}

\begin{proof}
    Direct adaptation from \cite[Proposition 3.2]{Sylla2024}.
\end{proof}

The theory of \cite{AKR2011} proposes an abstract framework for the study 
of \eqref{eq:CLDiscontinuous}. On the other hand, the authors of 
\cite{AGM2005, AJG2004} give a formula for the flux at the interface for Riemann 
problems of \eqref{eq:CLDiscontinuous}. In the following proposition, we link the two 
points of view. Let us adopt the notations:
\begin{equation}
    \label{eq:MinMax}
    \forall a, b \in \R, \quad a \wedge b \coloneq \min\{a, b\} 
    \quad \text{and} \quad a \vee b \coloneq \max\{a, b\}.
\end{equation}

\begin{proposition}
    \label{pp:FluxAJG}
    For all $(u_l, u_r) \in [0, 1]^2$, define the interface flux:
    \begin{equation}
        \label{eq:FluxAJG}
        \fint(u_l, u_r) \coloneq 
        \min\{f_l(u_l \wedge \alpha_l), f_r(\alpha_r \vee u_r)\},
    \end{equation}

    and the remainder term:
    \begin{equation}
        \label{eq:RemainderTerm}
        \cR(u_l, u_r) \coloneq 
        \abs{\fint(u_l, u_r) - f_l(u_l)} + \abs{\fint(u_l, u_r) - f_r(u_r)}.
    \end{equation}
    
    Then the following points hold.

    (i) For all $(u_l,u_r)\in [0, 1]^2$, $(u_l,u_r)\in \cG \iff \cR(u_l,u_r) = 0$.

    (ii) For all $(u_l, u_r) \in [0, 1]^2$ and for all $(k_l, k_r) \in \cG$, 
    $\Phi_r(u_r, k_r) - \Phi_l(u_l, k_l) \leq \cR(u_l, u_r)$. 

    (iii) For all $k \in [0, 1]$, $\cR(k, k) = \abs{f_l(k) - f_r(k)}$.
\end{proposition}

\begin{proof}
    \textit{(i)}-\textit{(ii)} follow from direct adaptations of \cite[Proposition 3.3]{Sylla2024} while \textit{(iii)} is obtained with straightforward computations.
\end{proof}

We are now in position to give the definition of solutions for \eqref{eq:CLDiscontinuous}. The first definition is a Kruzhkov type entropy inequality where the solution is tested not against constant data, but against piecewise constant data. The second definition relies on the treatment of the traces of the solution along $\{x=0\}$. More precisely, for $u \in \L{\infty}(\open{0}{+\infty} \times \R)$, the strong traces of $u$ along $\{x=0\}$, if they exist, are the functions $\gamma_{L}u, \gamma_{R}u \in \L{\infty}(\open{0}{+\infty}, \R)$ such that 
\[
    \underset{r \to 0+}{\mathrm{ess \; lim}} \; \frac{1}{r} \int_{-r}^0 \abs{u(t, x) - \gamma_{L}u(t)} \d{x} = 0, \quad \underset{r \to 0+}{\mathrm{ess \; lim}} \; \frac{1}{r} \int_0^r \abs{u(t, x) - \gamma_{R}u(t)} \d{x} = 0.
\]

\begin{theorem}
    \label{th:EquivalentDefinitions}
    Define $\cG$ as \eqref{eq:Germ} and $\cR$ as \eqref{eq:RemainderTerm}. 
    Let $u_o \in \L{\infty}(\R, \R)$ and $u \in \L{\infty}(\open{0}{+\infty} \times \R, \R)$. Then the statements $\mathbf{(A)}$ and $\mathbf{(B)}$ below are equivalent.

    $\mathbf{(A)}$ For all $\varphi \in \Cc{\infty}(\R^+ \times \R, \R^+)$ and for 
    all $(k_l, k_r) \in \R^2$, with $\kappa \coloneq k_l \1_{\open{-\infty}{0}} + k_r \1_{\open{0}{+\infty}}$, it holds:
    \begin{empheq}{align}
    \label{eq:DiscontinuousEI1}
    \int_0^{+\infty} \int_\R 
    \abs{u - \kappa(x)} \p_t \varphi + \Phi(x, u, \kappa(x)) \p_x \varphi 
    \; \d{x} \d{t} 
	+ \int_\R \abs{u_o(x) - \kappa(x)} \varphi(0, x) \d{x} \nonumber \\ 
	+ \int_0^{+\infty} \cR(k_l, k_r) \varphi(t, 0) \d{t} \geq 0.
    \end{empheq}

    $\mathbf{(B)}$ The two following points are verified.

    (i) For all $\varphi \in \Cc{\infty}(\R^+ \times \R \backslash \{0\}, \R^+)$ 
    and for all $(k_l, k_r) \in \R^2$, with $\kappa \coloneq k_l \1_{\open{-\infty}{0}} + k_r \1_{\open{0}{+\infty}}$, it holds:
    \begin{empheq}{align}
    \label{eq:DiscontinuousEI2}
    \int_0^{+\infty} \int_\R 
    \abs{u - \kappa(x)} \p_t \varphi + \Phi(x, u, \kappa(x))  \p_x \varphi \; \d{x} \d{t} 
    + \int_\R \abs{u_o(x) - \kappa(x)} \varphi(0, x) \d{x} \geq 0.
    \end{empheq}

    (ii) For a.e. $t \in \open{0}{+\infty}$, $(\gamma_L u (t), \gamma_R u (t)) \in \cG$.
    \smallskip 

    When $\mathbf{(A)}$ or $\mathbf{(B)}$ holds, we say that $u$ is an entropy solution to \eqref{eq:CLDiscontinuous} with initial datum $u_o$.
\end{theorem}

\begin{proof}
    See the proof of \cite[Theorem 3.18]{AKR2011}. 
    It is worth mentioning that either \eqref{eq:DiscontinuousEI1} or \eqref{eq:DiscontinuousEI2} implies that $u$ admits strong traces along $\{x=0\}$. Indeed, since $f_l$ and $f_r$ are genuinely nonlinear in the sense that
    \[
        \forall s \in \R, \quad \text{meas} \bigl(\left\{p \in \R \; : \; 
        f_{l,r}'(p) = s \right\} \bigr) = 0,
    \]

    existence of strong traces is obtained from the works of \cite{Vasseur2001, NPS2018}.
\end{proof}

We now state the well-posedness result for \eqref{eq:CLDiscontinuous}.

\begin{theorem}
    \label{th:WellPosednessDiscontinuous}
    Let $f_l, f_r \in \Lip([0, 1], \R^+)$ be two strictly concave, bell-shaped and ordered 
    functions. 
    \begin{enumerate}[label=\bf(D\arabic*)]
        \item \label{item:WellPosednessDiscontinuous} 
        Fix $u_o \in \L{\infty}(\R, [0, 1])$. 
        Then there exists a unique entropy solution to 
        \eqref{eq:CLDiscontinuous} with initial datum $u_o$.

        \item \label{item:StabilityDiscontinuous} 
        Let $u_o, v_o \in \L{\infty}(\R, [0, 1])$. 
        We denote by $u$, resp. $v$, the entropy solution to \eqref{eq:CLDiscontinuous} 
        with initial datum $u_o$, resp. $v_o$. Then, there exists $L > 0$ such that for 
        all $R > 0$ and for all $t > 0$,
        \begin{equation}
            \label{eq:StabilityDiscontinuous}
            \begin{aligned}
                \int_{\abs{x} \leq R} \abs{u(t, x) - v(t, x)} \d{x}
                & \leq \int_{\abs{x} \leq R+Lt} \abs{u_o(x) - v_o(x)} \d{x}.
            \end{aligned}
        \end{equation}
    \end{enumerate}
\end{theorem}

\begin{proof}
    The proof of the stability bound can be found in \cite[Theorem 3.11]{AKR2011}. 
    Regarding the existence, one can build a simple 
    finite volume scheme and prove its convergence, see for instance \cite{AJG2004}. 
    The flux to put at the interface $\{x = 0\}$ in the scheme is precisely $\fint$. 
\end{proof}

\subsection{Convergence of the numerical scheme}
\label{ssec:ProofScheme}

The proof of Theorem \ref{th:ConvergenceScheme} is split into several sections. The two following results will be of use.

\begin{lemma}
    \label{lmm:critical_point}
    Assume that $V_{\min}, V_{\max} \in \W{2}{\infty}(\open{0}{1}, \R^+)$ satisfy 
    \eqref{eq:VelocitiesAssumption}-\eqref{eq:StrictConcavity}. Then, there exists a unique function $\alpha \in \Lip([0, 1], [0, 1])$ such that for all $w\in [0, 1]$, 
    $\p_\rho f(\alpha(w), w) = 0$.
\end{lemma}

\begin{proof}
    Existence of the function $\alpha$ comes from the fact that $\p_\rho f(0, w) > 0$ and $\p_\rho f(1, w) < 0$, while the uniqueness follows from the monotonicity of $p \mapsto \p_\rho f(p, w)$. Moreover, as an application of the Implicit Function Theorem, we obtain, using \eqref{eq:StrictConcavity}:
    \[
        \abs{\alpha'(w)} = \abs{\frac{\p_{\rho w} f(\alpha(w), w)}{\p_{\rho}^2 f(\alpha(w), w)}}
        \leq \frac{1}{\eta} \sup_{p \in [0, 1]} \abs{f_{\max}'(p) - f_{\min}'(p)},
    \]

    which concludes the proof.
\end{proof}

Using Lemma \ref{lmm:critical_point}, one can rewrite the interface flux given by \eqref{eq:InterfaceFluxFirst} as
\[
    F_{j+1/2}^n (u, v) = \min \left\{ f_j^n (u \wedge \alpha_j^n), \;
    f_{j+1}^n (\alpha_{j+1}^n\vee v) \right\}, \quad \alpha_j^n \coloneq \alpha(w_j^n). 
\]

It is worth mentioning that $F_{j+1/2}^n$ is a monotone, locally Lipschitz, numerical flux, see \cite[Section 4]{AJG2004}. Also, remark that in general, $F_{j+1/2}^n$ is not consistent, meaning that for all $k\in [0, 1]$, $F_{j+1/2}^n(k,k)$ is not equal to $f_j^n (k)$ nor to $f_{j+1}^n (k)$. However, we can prove the following.

\begin{lemma}
    \label{lmm:ConsistencyInterfaceFlux}
    Assume that $V_{\min}, V_{\max} \in \W{2}{\infty}(\open{0}{1}, \R^+)$ satisfy 
    \eqref{eq:VelocitiesAssumption}-\eqref{eq:StrictConcavity}. Then, 
    \[
        \exists \mu > 0, \; 
        \forall k \in [0, 1], \; \forall n \in \N, \; \forall j \in \Z, \quad 
        \abs{F_{j+1/2}^n (k, k) - f_j^n (k)} \leq \mu \abs{w_{j+1}^n - w_j^n}. 
    \]
\end{lemma}

\begin{proof}
    Let $(n, j) \in \N \times \Z$ and $k \in [0, 1]$. To fix the ideas, suppose that 
    $f_{j}^n \geq f_{j+1}^n$, the case $f_j^n \leq f_{j+1}^n$ being entirely similar. Recall that
    \[
        F_{j+1/2}^n(k, k) =
        \left\{ 
            \begin{array}{cclc}
            \min\{ f_j^n (k), f_{j+1}^n (\alpha_{j+1}^n) \} & \text{if} 
            & k \leq \alpha_j^n, \alpha_{j+1}^n \quad & \text{(I)} \\
            f_{j+1}^n (\alpha_{j+1}^n) & \text{if} 
            & \alpha_{j}^n \leq k \leq \alpha_{j+1}^n \quad & \text{(II)} \\
            f_{j+1}^n (k) & \text{if} 
            & k \geq \alpha_{j+1}^n  \quad & \text{(III)}
            \end{array}
        \right.
    \]

    The proof reduces to a case by case study. Set 
    $\delta \coloneq \abs{F_{j+1/2}^n(k, k) - f_{j}^n (k)}$.

    \textbf{Case (I)} If $F_{j+1/2}^n(k, k) = f_{j}^n(k)$, then $\delta = 0$. 
    Otherwise, $F_{j+1/2}^n(k, k) = f_{j+1}^n (\alpha_{j+1}^n)$, meaning that 
    $f_{j}^n(k) \geq f_{j+1}^n (\alpha_{j+1}^n)$. Therefore,
    \[
        \delta = f_j^n (k) - f_{j+1}^n (\alpha_{j+1}^n)
        \leq f_j^n(k) - f_{j+1}^n (k)
        = \abs{w_{j+1}^n - w_j^n} \; (f_{\max}(k) - f_{\min}(k)).
    \]

    \textbf{Case (II)} In that case, the estimate follows from the fact that 
    \[
        \begin{aligned}
            \abs{f_{j+1}^n (\alpha_{j+1}^n) - f_j^n (k)}
            & \leq \abs{f_{j+1}^n (\alpha_{j+1}^n) - f_{j+1}^n (k)} + \abs{f_{j+1}^n (k) - f_j^n (k)} \\
            & \leq \abs{\alpha_{j+1}^n - \alpha_j^n} \sup_{0 \leq p, \omega \leq 1} 
            \abs{\p_\rho f(p, \omega)} + \abs{w_{j+1}^n - w_j^n} \; (f_{\max}(k) - f_{\min}(k)) \\
            & \leq \left( \norm{\alpha'}_{\L{\infty}([0, 1])} \sup_{0 \leq p, \omega \leq 1} \abs{\p_\rho f(p, \omega)} + (f_{\max}(k) - f_{\min}(k)) \right) \abs{w_{j+1}^n - w_j^n}.
        \end{aligned}
    \]

    \textbf{Case (III)} In that case, we can bound $\delta$ by 
    $\abs{w_{j+1}^n - w_j^n} \; (f_{\max}(k) - f_{\min}(k))$.

    Therefore, the following constant is suitable:
    \[
        \mu \coloneq \norm{\alpha'}_{\L{\infty}([0, 1])} \sup_{0 \leq p, \omega \leq 1} \abs{\p_\rho f(p, \omega)} + \sup_{p\in [0, 1]} \abs{f_{\max}(p) - f_{\min}(p)}.
    \]
\end{proof}

It will be useful to write \eqref{eq:MF2} under the form 
\[
    \rho_j^{n+1} = H_j^n(\rho_{j-1}^n, \rho_j^n, \rho_{j+1}^n),
\]

where $H_j^n = H_j^n(a, b, c)$ is defined with the right-hand side of \eqref{eq:MF2}.

\subsubsection{Stability and BV propagation}

\begin{lemma}
    \label{lmm:StabilityScheme}
    Assume that $V_{\min}, V_{\max} \in \W{2}{\infty}(\open{0}{1}, \R^+)$ satisfy 
    \eqref{eq:VelocitiesAssumption}-\eqref{eq:vacuum}. Fix 
    $\rho_o \in \L{\infty}(\R, [\eps, 1])$ and $w_o \in \L{\infty}(\R, [0, 1])$. 
    Then under the CFL condition \eqref{eq:CFL}, for all $n \in \N$, we have:
    \begin{equation}
        \label{eq:StabilityScheme}
        \forall j \in \Z, \quad \eps \leq \rho_j^n \leq 1 \quad \text{and} \quad 
        0 \leq w_{j}^{n} \leq 1.
    \end{equation}    
\end{lemma}

\begin{proof}
    Clearly, \eqref{eq:StabilityScheme} holds for $n=0$. Suppose that it holds for some $n\in \N$. 
    
    (1) Using \eqref{eq:CFL}, we verify that the scheme is monotone. More precisely, 
    for a.e. $a, b, c \in [\eps, 1]$, 
    \[
        \p_1 H_j^n(a, b, c) = \lambda \p_1 F_{j-1/2}^n (a, b) \geq 0, \quad
        \p_3 H_j^n(a, b, c) = -\lambda \p_2 F_{j+1/2}^n (b, c) \geq 0,
    \]

    and 
    \[
        \p_2 H_j^n(a, b, c) 
        = 1 - \lambda \left( \p_1 F_{j+1/2}^n (b, c) - \p_2 F_{j-1/2}^n (a, b)\right) 
        \geq 1 - 2 \lambda L \geq 0.
    \]

    Consequently, on the one hand, using \eqref{eq:vacuum}, 
    \[
        \rho_j^{n+1} \geq H_j^n(\eps, \eps, \eps) 
        = \eps - \lambda (f_{j+1}^n (\eps) - f_j^n(\eps)) = \eps.
    \]

    On the other hand, $\rho_j^{n+1} \leq H_j^n(1, 1, 1) = 1$, therefore $\rho_j^{n+1} \in [\eps, 1]$.
    
    (2) Due to the CFL condition, $w_{j}^{n+1}$ is a convex combination of $w_{j}^n$ and $w_{j-1}^n$, therefore $w_{j}^{n+1}\in [0,1]$, which completes the induction argument.
\end{proof}

\bigskip 

Global $\BV$ estimates can be derived for $(w_\Delta)_\Delta$.

\begin{proposition}
\label{pp:Compactness1}
    Assume that $V_{\min}, V_{\max} \in \W{2}{\infty}(\open{0}{1}, \R^+)$ satisfy 
    \eqref{eq:VelocitiesAssumption}-\eqref{eq:vacuum}.
    Fix $\rho_o \in \L{\infty}(\R, [\eps, 1])$ and $w_o \in \BV(\R, [0, 1])$. 
    Assume that \eqref{eq:CFL} holds. Then for all $T, X > 0$, 
    the following $\BV$ estimates hold:
    \begin{equation}
    \label{eq:BVBounds1}
        \forall t \in [0, T], \quad \TV(w_\Delta(t^n)) \leq \TV(w_o)
    \end{equation}

    and
    \begin{equation}
    \label{eq:BVBounds2}
        \forall t \in [0, T-\Delta t], \quad \int_{\abs{x} \leq X} \abs{w_\Delta(t+\Delta t, x) - w_\Delta(t, x)} \d{x} \leq 
        \frac{L}{\eps} \TV(w_o) \Delta t.
    \end{equation}

    Consequently, there exists $w \in \Czero(\R^+, \Lloc{1}(\R))\cap \Lloc{\infty}(\R^+, \BV(\R))$ such that, up to the extraction of a subsequence, $w_\Delta \to w$ in $\Czero(\R^+,\Lloc{1}(\R))$.
\end{proposition}

\begin{proof}
    Let $n \in \N$ and $j \in \Z$. First, write 
    \[
        \begin{aligned}
            w_{j}^{n+1} - w_{j-1}^{n+1} 
            & = (1 - \lambda s_{j-1/2}^{n}) (w_j^n - w_{j-1}^n) 
            + \lambda s_{j-3/2}^{n} (w_{j-1}^n - w_{j-2}^n).
        \end{aligned}
    \]

    Therefore, using the CFL condition, we derive:
    \[
        \begin{aligned}
            \TV(w_\Delta(t^{n+1})) 
            & \leq \sum_{j \in \Z} \abs{w_{j}^{n} - w_{j-1}^{n}}.
        \end{aligned}
    \]

    Then, using the marching formula \eqref{eq:MFw}, for all $n \in \N$ and $J \in \N \setminus \{0\}$,
    \[
        \begin{aligned}
            \sum_{\abs{j} \leq J} \abs{w_{j}^{n+1} - w_{j}^{n}} \Delta x
            & \leq \frac{L}{\eps}\TV(w_\Delta (t^n)) \Delta t 
            \leq \frac{L}{\eps} \TV(w_o) \Delta t.
        \end{aligned}
    \]

    Combining \eqref{eq:BVBounds1}-\eqref{eq:BVBounds2} with \eqref{eq:StabilityScheme}, \cite[Theorem 5.13]{SerreBookVol1} proves the compactness statement.
\end{proof}

\subsubsection{Compensated compactness}
\label{sssec:CompComp}

We use the compensated compactness method on $(\rho_\Delta)_\Delta$, whose  
applications to systems of conservation laws are for instance reviewed in 
\cite{Chen2000, LuBook}. Modified for our use, the compensated compactness lemma reads 
as follows.

\begin{lemma}
    \label{lmm:CompComp}
    Assume that $V_{\min}, V_{\max} \in \W{2}{\infty}(\open{0}{1}, \R^+)$ satisfy 
    \eqref{eq:VelocitiesAssumption}-\eqref{eq:StrictConcavity}.
    Fix $w \in \Lloc{\infty}(\R^+, \BV(\R))$ 
    and let $(u_n)_n$ be a bounded sequence of 
    $\L{\infty}(\open{0}{+\infty} \times \R, \R)$ such that for all 
    $k \in \R$ and for any $i \in \{1, 2\}$, the sequence 
    $(\p_t S_i (u_n, w) + \p_x Q_i (u_n, w))_n$ belongs to a compact subset of 
    $\Hloc{-1}(\open{0}{+\infty} \times \R, \R)$, where 
    \begin{equation}
        \label{eq:CompCompEntropies}
        \begin{array}{cc}
            S_1(u) \coloneq u - k 
            & Q_1(u, \omega) \coloneq f(u, \omega) - f(k, \omega) \\[5pt]
            S_2(u, \omega) \coloneq f(u, \omega) - f(k, \omega) 
            & \ds{Q_2(u, \omega) \coloneq \int_k^u \p_\rho f(\xi, \omega)^2 \d{\xi}}.
        \end{array} 
    \end{equation}
    
    Then there exists a subsequence of $(u_n)_n$ that converges in 
    $\Lloc{p}(\open{0}{+\infty} \times \R, \R)$ for all 
    $p \in [1, +\infty \mathclose[$ and a.e. on 
    $\open{0}{+\infty} \times \R$ to some function 
    $u \in \L{\infty}(\open{0}{+\infty} \times \R, \R)$.
\end{lemma}

To prove the $\Hloc{-1}$ compactness required in Lemma \ref{lmm:CompComp} we will use the following technical result, see \cite{Murat1981, DCP1987, SerreBook}.

\begin{lemma}
    \label{lmm:CompCompTechnical}
    Let $q, r \in \R$ such that $1 < q < 2 < r$. Let $(\mu_n)_n$ be a sequence of distributions such that:
    
    (i) $(\mu_n)_n$ belongs to a compact subset of $\Wloc{-1}{q}(\open{0}{+\infty} \times \R, \R)$.

    (ii) $(\mu_n)_n$ is bounded in $\Wloc{-1}{r}(\open{0}{+\infty} \times \R, \R)$.

    Then $(\mu_n)_n$ belongs to a compact subset of $\Hloc{-1}(\open{0}{+\infty} \times \R, \R)$.
\end{lemma}

For the compactness analysis, we take inspiration from 
\cite{KMRTriangular, KMRWellBalanced}.

Let $(S, Q)$ be a smooth entropy/entropy flux pair, which we recall, means that 
\[
    \forall \omega, u \in [0, 1], \quad \p_u Q(u, \omega) = S'(u) \; \p_\rho f(u, \omega).
\]

The entropy dissipation of $(\rho_\Delta, w_\Delta)$ is defined as 
\begin{equation}
    \label{eq:EntropyDissipation}
    E_\Delta(\phi) \coloneq \int_{0}^{+\infty} \int_{\R} 
    S(\rho_\Delta) \p_t \phi + Q(\rho_\Delta, w_\Delta)  \p_x \phi \; \d{x} \d{t}, 
    \quad \phi \in \Cc{\infty}(\R^+ \times \R, \R).
\end{equation}

Fix $n \in \N$ and $j \in \Z$. Let us consider the entropy dissipation in 
$P_{j+1/2}^n \coloneq [t^n, t^{n+1} \mathclose[ \times \open{x_j}{x_{j+1}}$. By integration by parts and the fact that $\rho_\Delta$ is the exact solution of a 
Riemann problem in $P_{j+1/2}^n$, we write:
\[
    \begin{aligned}
        \iint_{P_{j+1/2}^n} 
        & S(\rho_\Delta) \p_t \phi + Q(\rho_\Delta, w_\Delta) \p_x \phi \;\d{x} \d{t} \\
        & = \int_{x_j}^{x_{j+1}} S(\rho_\Delta(t^{n+1}-, x)) \phi(t^{n+1}, x) 
        -  S(\rho_\Delta(t^{n}, x)) \phi(t^{n}, x) \; \d{x} \\
        & + \int_{t^{n}}^{t^{n+1}} Q(\rho_{j+1}^n, w_{j+1}^n) \phi(t, x_{j+1}) 
        - Q(\rho_j^n, w_{j}^n) \phi(t, x_{j}) \; \d{t} \\
        & + \int_{t^{n}}^{t^{n+1}} [\![ \sigma S(\rho_\Delta) - 
        Q(\rho_\Delta, w_\Delta) ]\!]_y \; \phi(t, y(t)) \d{t} \\
        & + \int_{t^{n}}^{t^{n+1}} \biggl( Q(\rho_\Delta(t, x_{j+1/2}-), w_j^n) 
        - Q(\rho_\Delta(t, x_{j+1/2}+), w_{j+1}^n) \biggr) \phi(t, x_{j+1/2}) \d{t}.
    \end{aligned}
\]

In the right-hand side of the previous equality: 

$\bullet$ The third integral is the contribution of the (eventual) classical shock in 
the solution. Hence, $\sigma$ is the shock velocity, given by the Rankine-Hugoniot 
condition, $y$ is the shock curve and 
\[
    [\![ \sigma S(\rho_\Delta) - Q(\rho_\Delta, w_\Delta) ]\!]_y
    \coloneq (\sigma S(\rho_\Delta) - Q(\rho_\Delta, w_\Delta))(t, y(t)+) 
    - (\sigma S(\rho_\Delta) - Q(\rho_\Delta, w_\Delta))(t, y(t)-).
\]

It is relevant to say that $w_\Delta$ is continuous along $\{x=y(t)\}$, equal to 
either $w_j^n$ or $w_{j+1}^n$, depending on the sign of $\sigma$.

$\bullet$ The last integral is the contribution of the (eventual) non-classical shock 
in the solution.

We refer to Section \ref{ssec:DiscontinuousFlux} where the structure of solutions to Riemann Problems is discussed.

\bigskip

Taking the sum for $n \in \N$ and $j \in \Z$, we see that the entropy dissipation rewrites as
\begin{equation}
    \label{eq:EntropyDissipation2}
    \begin{array}{clr}
    E_\Delta(\phi) 
    & \ds{= -\int_{\R} S(\rho_\Delta(0, x)) \phi(0, x) \d{x}} 
    & \longleftarrow I_1(\phi) \\[5pt]
    & \ds{+ \sum_{n=1}^{+\infty} \int_{\R} 
    (S(\rho_\Delta(t^{n}-, x)) - S(\rho_\Delta(t^{n}, x))) \phi(t^n, x) \; \d{x}} 
    & \longleftarrow I_2(\phi) \\[5pt]
    & \ds{+ \sum_{n=0}^{+\infty} \sum_{j \in \Z} \sum_{y} \int_{t^{n}}^{t^{n+1}} 
    [\![ \sigma S(\rho_\Delta) - Q(\rho_\Delta, w_\Delta) ]\!]_y \; 
    \phi(t, y(t)) \d{t}} 
    & \longleftarrow I_3(\phi) \\[5pt]
    & \ds{+ \sum_{n=0}^{+\infty} \sum_{j \in \Z} \int_{t^{n}}^{t^{n+1}} 
    \biggl( Q(\rho_\Delta(t, x_{j+1/2}-), w_j^n) 
    - Q(\rho_\Delta(t, x_{j+1/2}+), w_{j+1}^n) \biggr) \phi(t, x_{j+1/2}) \; \d{t}} 
    & \longleftarrow I_4(\phi) \\[5pt]
\end{array}
\end{equation}

As we previously explained, for all $j \in \Z$, the sum $\sum_{y}$ in 
\eqref{eq:EntropyDissipation2} is either over an empty set, or over a singleton.
\bigskip

First, let us give a bound on the variation of the approximate solution across the 
discrete time levels.

\begin{lemma}
    \label{lmm:CompCompA}
    Let $T > 0$ and $X > 0$. Fix $N, J \in \N$ such that 
    $T \in [t^N, t^{N+1} \mathclose[$ and $X \in I_{J}$. Then, there exists a constant $c_1 > 0$ depending on $T$, $X$, $\rho_o$, $w_o$ and $f$ such that 
    \begin{equation}
        \label{eq:CompCompA}
        \begin{aligned}
            & \sum_{n=1}^{N} \sum_{\abs{j} \leq J} 
        \int_{I_j} \abs{\rho_\Delta (t^n, x) - \rho_\Delta(t^n-, x)}^2 \d{x} \leq c_1 \\
        \text{and} \quad & 
        \sum_{n=0}^{N} \sum_{\abs{j} \leq J} \sum_{y} \int_{t^{n}}^{t^{n+1}} 
        [\![ \sigma S(\rho_\Delta) - Q(\rho_\Delta, w_\Delta) ]\!]_y \; \d{t} \leq c_1.
        \end{aligned}
    \end{equation}
\end{lemma}

\begin{proof}
    Choose $S(u) \coloneq \frac{u^2}{2}$ and $Q$ its entropy flux. Let $(\phi_k)_k$ be a sequence of nonnegative test functions that converges to $\phi \coloneq \1_{[0, T] \times [-X, X]}$. At the limit $k \to +\infty$ in \eqref{eq:EntropyDissipation2}, we have
    \[
        I_2(\phi) + I_3(\phi) = \int_{\abs{x} \leq X} \abs{\rho_\Delta(0, x)}^2 \d{x} - I_4(\phi).
    \]  

    Clearly, 
    \[
        \int_{\abs{x} \leq X} \abs{\rho_\Delta(0, x)}^2 \d{x} 
        \leq 2 X \norm{\rho_o}_{\L{\infty}(\R)}^2,
    \]

    and $I_3(\phi)$ is nonnegative because the (eventual) shock is classical and 
    therefore produces entropy. 

    Then, notice that by definition of $(\rho_j^n)_{j, n}$,
    \[
        \begin{aligned}
            I_2(\phi) 
            & = \frac{1}{2} \sum_{n=1}^{N} \sum_{\abs{j} \leq J} \int_{I_j} \left(
            \rho_\Delta(t^{n}-, x)^2 - \rho_\Delta(t^{n}, x)^2 \right) \; \d{x} \\ 
            & = \frac{1}{2} \sum_{n=1}^{N} \sum_{\abs{j} \leq J} \int_{I_j} \biggl\{ 
            (\rho_\Delta(t^{n}-, x) - \rho_\Delta(t^{n}, x))^2
            +2 \rho_\Delta(t^{n}, x) (\rho_\Delta(t^{n}-, x)  - \rho_\Delta(t^{n}, x))
            \biggr\} \d{x} \\
            & = \frac{1}{2} \sum_{n=1}^{N} \sum_{\abs{j} \leq J} \int_{I_j} 
            (\rho_\Delta(t^{n}-, x) - \rho_\Delta(t^{n}, x))^2 \d{x}
            + \sum_{n=1}^{N} \sum_{\abs{j} \leq J} \rho_j^n \underbrace{\int_{I_j}
            (\rho_\Delta(t^{n}-, x)  - \rho_j^n) \d{x}}_{=0}.
        \end{aligned}
    \]

    Finally, as a consequence of Proposition \ref{pp:FluxAJG} \textit{(ii)}, for all 
    $k \in [0, 1]$,
    \[
        \Phi(\rho_\Delta(t, x_{j+1/2}-), w_j^n, k) - \Phi(\rho_\Delta(t, x_{j+1/2}+), w_{j+1}^n, k)
        \geq - \cR_{j+1/2}^n (k, k).
    \]

    From Proposition \ref{pp:FluxAJG} \textit{(iii)}, we bound this remainder term as following:
    \[
        \cR_{j+1/2}^n (k, k) 
        = \abs{f_{j+1}^n (k) - f_j^n(k)} 
        \leq \sup_{p \in [\eps, 1]} \abs{f_{\min}(p) - f_{\max}(p)} \cdot \abs{w_{j+1}^n - w_j^n}.
    \]
    
    Using an approximation argument to pass from Kruzhkov entropies to any entropy, 
    see \cite{KMRTriangular}, we obtain
    \[
        I_4(\phi) 
        \geq - T \sup_{p \in [\eps, 1]} \abs{f_{\min}(p) - f_{\max}(p)} \; \TV(w_\Delta(t^n)).
    \]

    Estimate \eqref{eq:CompCompA} follows by putting the bounds on $I_1(\phi)$ and 
    $I_4(\phi)$ together and using Proposition \ref{pp:Compactness1}.
\end{proof}

We can convert \eqref{eq:CompCompA} into an estimate on the spatial variation of the 
approximate solutions, following the arguments of \cite[Page 67]{DiPerna1983}. 
For the sake of clarity, we set $\rho_{j+1/2}^{n, \pm} \coloneq \rho_\Delta(t^n, x_{j+1/2} \pm)$.

\begin{lemma}
    \label{lmm:CompCompB}
    Let $T > 0$ and $X > 0$. Fix $N, J \in \N$ such that 
    $T \in [t^N, t^{N+1} \mathclose[$ and $X \in I_{J}$. Then, there exists a constant $c_2> 0$ depending on $T$, $X$, $\rho_o$, $w_o$ and $f$ such that 
    \begin{equation}
        \label{eq:CompCompB}
        \sum_{n=0}^{N} \sum_{\abs{j} \leq J} \left\{ 
            (\rho_{j+1/2}^{n, -} - \rho_{j+1/2}^{n, +})^2 
            + (\rho_j^{n} - \rho_{j+1/2}^{n, -})^2 
            + (\rho_{j+1}^{n} - \rho_{j+1/2}^{n, +})^2 \right\} \Delta x \leq c_2. 
    \end{equation}
\end{lemma}

We are now in position to prove the $\Hloc{-1}$ compactness of the sequence of measures 
that defines the entropy dissipation.

\begin{lemma}
    \label{lmm:CompCompC}
    Assume that $V_{\min}, V_{\max} \in \W{2}{\infty}(\open{0}{1}, \R^+)$ satisfy 
    \eqref{eq:VelocitiesAssumption}-\eqref{eq:vacuum}.
    Fix $\rho_o \in \L{\infty}(\R, [\eps, 1])$ and $w_o \in \BV(\R, [0, 1])$.
    Let $w$ be the limit function of Proposition \ref{pp:Compactness1} and let 
    $(S_i, Q_i)_{i \in \{1, 2\}}$ be the entropy/entropy flux pairs defined in 
    Lemma \ref{lmm:CompComp}. Then for any $i \in \{1, 2\}$, the sequence of 
    distributions $\mu_i$ defined by 
    \[
        \mu_i(\phi) \coloneq \int_{0}^{+\infty} \int_{\R} 
        S_i(\rho_\Delta, w) \p_t \phi + Q_i(\rho_\Delta, w)  \p_x \phi \; \d{x} \d{t}, 
        \quad \phi \in \Cc{\infty}(\R^+ \times \R, \R),
    \]
    
    belongs to a compact subset of $\Hloc{-1}(\open{0}{+\infty} \times \R, \R)$.
\end{lemma}

\begin{proof}
    We follow the proofs of \cite[Lemma 5.5]{KMRTriangular} or 
    \cite[Lemma 4.5]{KMRWellBalanced}. To start, consider $(S, Q)$ a smooth 
    entropy/entropy flux pair.

    Let us fix a bounded open subset 
    $\Omega \subset \open{0}{+\infty} \times \R$, say 
    $\Omega \subset [0, T] \times [-X, X]$ for some $T > 0$ and $X > 0$. 
    Let $N, J \in \N \setminus \{0\}$ such that $T \in [t^N, t^{N+1} \mathclose[$ and 
    $X \in I_{J}$, and finally, let $\phi \in \Cc{\infty}(\Omega, \R)$. We split $\mu$ as 
    \[
        \mu(\phi) = 
        \int_{0}^{+\infty} \int_{\R} (Q(\rho_\Delta, w) - Q(\rho_\Delta, w_\Delta)) 
        \p_x \phi \; \d{x} \d{t} + E_\Delta(\phi),
    \]

    where $E_\Delta(\phi)$ is given by \eqref{eq:EntropyDissipation}. 

    The strong compactness established in Proposition \ref{pp:Compactness1} ensures 
    that for any $q_1 \in \mathopen]1, 2]$, the first term of the right-hand side is 
    strongly compact in $\W{-1}{q_1}(\Omega, \R)$. We now estimate $E_\Delta(\phi)$.

    First, observe that since $(\rho_\Delta)_\Delta$ and $(w_\Delta)_\Delta$ are 
    bounded in $\L{\infty}(\Omega, \R)$, we have 
    \[
        \abs{E_\Delta(\phi)} \leq \biggl(\sup_{p \in [\eps, 1]} \abs{S(p)} 
        + \sup_{p, \omega \in [0, 1]} \abs{Q(p, \omega)} \biggr) \cdot 
        \norm{\phi}_{\W{1}{\infty}(\Omega)},    
    \]

    implying that $(E_\Delta(\phi))_\Delta$ is bounded in $\Wloc{-1}{r}(\Omega, \R)$ 
    for any $r \in \open{2}{+\infty}$.

    Keeping the notations of \eqref{eq:EntropyDissipation2}, we bound $I_1(\phi)$ and 
    $I_3(\phi)$ as follows:
    \[
        \abs{I_1(\phi)} 
        \leq \norm{S(\rho_o)}_{\L{1}([-X, X])} \; \norm{\phi}_{\L{\infty}(\Omega)}, 
        \quad I_3(\phi) \leq c_1 \norm{\phi}_{\L{\infty}(\Omega)},
    \]

    where we used Lemma \ref{lmm:CompCompA} for $I_3(\phi)$. Consequently, 
    $I_1(\phi)$ and $I_3(\phi)$ are bounded in the space $\cM(\Omega, \R)$ of 
    bounded Radon measures, which is compactly embedded in $\W{-1}{q_2}(\Omega, \R)$ if 
    $q_2 \in \open{1}{2}$. Therefore, $I_1(\phi)$ and $I_3(\phi)$ belong to a compact subset of 
    $\W{-1}{q_2}(\Omega, \R)$.

    Now, to estimate $I_2(\phi)$, split it as $I_2(\phi) = I_{2,1} (\phi) + I_{2, 2}(\phi)$ where 
    \[
        \begin{aligned}
            I_{2, 1}(\phi) & \coloneq \sum_{n=1}^{N} \sum_{\abs{j} \leq J} \int_{I_j} 
            (S(\rho_\Delta(t^{n}-, x)) - S(\rho_j^n)) \; \phi(t^n, x_j) \d{x} \\
            I_{2, 2}(\phi) & \coloneq \sum_{n=1}^{N} \sum_{\abs{j} \leq J} \int_{I_j} 
            (S(\rho_\Delta(t^{n}-, x)) - S(\rho_j^n)) \;
            (\phi(t^n, x) - \phi(t^n, x_j)) \d{x}.
        \end{aligned}
    \]

    We handle $I_{2,1}(\phi)$ by introducing an intermediary point $\chi_j^n$ such that 
    \[
        S(\rho_\Delta(t^{n}-, x)) - S(\rho_j^n) 
        = S'(\rho_j^n) (\rho_\Delta(t^{n}-, x) - \rho_j^n) 
        + \frac{S''(\chi_j^n)}{2} (\rho_\Delta(t^{n}-, x) - \rho_j^n)^2.
    \]

    Taking into account the definition of $\rho_j^n$ and Lemma \ref{lmm:CompCompA}, we 
    obtain
    \[
        \abs{I_{2, 1}(\phi)} \leq \frac{c_1}{2} \sup_{p \in [\eps, 1]} S''(p) \cdot
        \norm{\phi}_{\L{\infty}(\Omega)},
    \]

    ensuring that $I_{2, 1}(\phi)$ belongs to a compact subset of $\W{-1}{q_2}(\Omega, \R)$.

    We continue by choosing $\alpha \in \open{\frac{1}{2}}{1}$ and then writing:
    \[
        \begin{aligned}
            \abs{I_{2,2}(\phi)} 
            & \leq \norm{\phi}_{\Ck{0, \alpha}(\Omega)} \Delta x^\alpha
            \sum_{n=1}^{N} \sum_{\abs{j} \leq J} \int_{I_j} 
            \abs{S(\rho_\Delta(t^{n}-, x)) - S(\rho_j^n)} \d{x}  \\
            & \leq \norm{\phi}_{\Ck{0, \alpha}(\Omega)} \Delta x^{\alpha-1}
            \left\{ \sum_{n=1}^{N} \sum_{\abs{j} \leq J} \left(\int_{I_j} 
            \abs{S(\rho_\Delta(t^{n}-, x)) - S(\rho_j^n)} \d{x} \right)^2 \right\}^{1/2}
            \left\{ \sum_{n=1}^{N} \sum_{\abs{j} \leq J} \Delta x^2 \right\}^{1/2} \\
            & \leq \sqrt{\frac{2TX}{\lambda}} \norm{\phi}_{\Ck{0, \alpha}(\Omega)} \Delta x^{\alpha-1/2}
            \left\{ \sum_{n=1}^{N} \sum_{\abs{j} \leq J} \int_{I_j} 
            \abs{S(\rho_\Delta(t^{n}-, x)) - S(\rho_j^n)}^2 \d{x} \right\}^{1/2} \\
            & \leq \sqrt{\frac{2 T X c_1}{\lambda}} 
            \sup_{p \in [\eps, 1]} \abs{S'(p)} \cdot 
            \norm{\phi}_{\Ck{0, \alpha}(\Omega)} \; \Delta x^{\alpha-1/2} \qquad
            \text{Lemma \ref{lmm:CompCompA}}
        \end{aligned}
    \]

    We now take advantage of the compact embedding 
    $\Ck{0, \alpha}(\Omega, \R) \subset \W{1}{p}(\Omega, \R)$, $p=2/(1-\alpha)$ to be 
    sure that $I_{2,2} (\phi)$ belongs to a compact subset of $\W{-1}{q_3}(\Omega, \R)$, $q_3\coloneq 2/(1+\alpha) \in \open{1}{\frac{4}{3}}$.

    Finally, we estimate $I_4(\phi)$. Let $k \in [0, 1]$ and consider 
    $(S_i, Q_i)_{i \in \{1, 2\}}$ the entropy/entropy flux pairs defined in 
    Lemma \ref{lmm:CompComp}. Because of the Rankine-Hugoniot condition,
    \[
        \abs{[\![ Q_1(\rho_\Delta, w_\Delta) ]\!]_{x_{j+1/2}}^n}
        = \abs{f_j^n (k) - f_{j+1}^n (k)},
    \]

    from which we deduce,
    \[
        \begin{aligned}
            \abs{I_{4,1}(\phi)}
            & \leq \sum_{n=0}^{N} \sum_{\abs{j} \leq J} \int_{t^{n}}^{t^{n+1}} 
            \abs{f_j^n (k) - f_{j+1}^n (k)} \cdot \abs{\phi(t, x_{j+1/2})} \d{t} \\
            & \leq \sup_{p \in [\eps, 1]} \abs{f_{\min}(p) - f_{\max}(p)} 
            \cdot \norm{\phi}_{\L{\infty}(\Omega)} \sum_{n=0}^{N} 
            \TV(w_\Delta(t^n)) \; \Delta t,
        \end{aligned}
    \]

    ensuring, using once more Proposition \ref{pp:Compactness1}, that $I_{4, 1}(\phi)$ belongs to a compact subset of $\W{-1}{q_2}(\Omega, \R)$. 

    On the other hand,
    \[
        \begin{aligned}
            \abs{[\![ Q_2(\rho_\Delta, w_\Delta) ]\!]_{x_{j+1/2}}^n}
            & = \int_{\rho_{j+1/2}^{n,+}}^{\rho_{j+1/2}^{n,-}} 
            \p_\rho f(\xi, w_j^n)^2 \d{\xi} 
            + \int_{k}^{\rho_{j+1/2}^{n,+}} \p_\rho f(\xi, w_j^n)^2
            - \p_\rho f(\xi, w_{j+1}^n)^2 \d{\xi} \\
            & \leq \int_{\rho_{j+1/2}^{n,+}}^{\rho_{j+1/2}^{n,-}} 
            \p_\rho f(\xi, w_j^n)^2 \d{\xi} + (2L)^2 \abs{w_{j+1}^n-w_j^n} \\
            & \leq L \underbrace{\int_{\rho^{-}}^{\rho^{+}} 
            \abs{\p_\rho f(\xi, w_j^n)} \d{\xi}}_{Q_{2,1}} + (2L)^2 \abs{w_{j+1}^n-w_j^n}, 
            \quad \rho^{-} \coloneq \min\left\{\rho_{j+1/2}^{n,-}, \rho_{j+1/2}^{n,+} \right\}.
        \end{aligned}
    \] 

    If $(f_j^n)'$ does not change sign in $\open{\rho^-}{\rho^+}$, then 
    \[
        \begin{aligned}
            \abs{Q_{2, 1}} 
            & = \abs{f_j^n (\rho_{j+1/2}^{n,-}) - f_j^n (\rho_{j+1/2}^{n,+})} \\
            & \leq \underbrace{\abs{f_j^n (\rho_{j+1/2}^{n,-}) 
            - f_{j+1}^n (\rho_{j+1/2}^{n,+})}}_{=0} 
            + \abs{f_{j+1}^n (\rho_{j+1/2}^{n,+}) - f_j^n (\rho_{j+1/2}^{n,+})} \\
            & \leq \sup_{p \in [\eps, 1]} \abs{f_{\min}(p) - f_{\max}(p)}
            \cdot \abs{w_{j+1}^n - w_j^n}. 
        \end{aligned}
    \]

    Otherwise, we can assume that $\rho^- \leq \alpha_j^n \leq \rho^+$. 
    Then, setting $L_f \coloneq \sup_{p, \omega \in [0, 1]} \abs{\p_{\rho}^2 f(p, \omega)}$, we write
    \[
        \begin{aligned}
            \int_{\rho^{-}}^{\rho^{+}} \abs{\p_\rho f(\xi, w_j^n) - 
            \p_\rho f(\alpha_j^n, w_j^n)} \d{\xi}
            \leq L_f \int_{\rho^{-}}^{\rho^{+}} \abs{\xi - \alpha_j^n} \d{\xi}
            \leq L_f (\rho^+ - \rho^-)^2
            = L_f \; (\rho_{j+1/2}^{n,-} - \rho_{j+1/2}^{n,+})^2.
        \end{aligned}
    \]

    Taking advantage of Lemma \ref{lmm:CompCompB}, we conclude that 
    \[
        \begin{aligned}
            & \sum_{n=0}^{N} \Delta t \sum_{\abs{j} \leq J}  
            \abs{[\![ Q_2(\rho_\Delta, w_\Delta) ]\!]_{x_{j+1/2}}^n} 
            \cdot \norm{\phi}_{\L{\infty}(\Omega)} \\
            & \leq L \left\{L_f + \sup_{p \in [\eps, 1]} \abs{f_{\min}(p) - f_{\max}(p)} 
            + 4L \right\} \sum_{n=0}^{N} \Delta t \sum_{\abs{j} \leq J}  
            \left\{ (\rho_{j+1/2}^{n,-} - \rho_{j+1/2}^{n,+})^2 
            + \abs{w_{j+1}^n-w_j^n} \right\} \cdot \norm{\phi}_{\L{\infty}(\Omega)} \\
            & \leq L \left\{L_f + \sup_{p \in [\eps, 1]} \abs{f_{\min}(p) - f_{\max}(p)} 
            + 4L \right\} (\lambda c_2 + \sum_{n=0}^N \TV(w_\Delta(t^n)) \Delta t) \cdot \norm{\phi}_{\L{\infty}(\Omega)},
        \end{aligned}
    \]

    which ensures $I_{4, 2}(\phi)$ belongs to a compact subset of 
    $\W{-1}{q_2}(\Omega, \R)$. 

    To summarize, for any $i \in \{1, 2\}$, $(\mu_i)_\Delta$ belongs to a compact 
    subset of $\Wloc{-1}{q}(\open{0}{+\infty} \times \R, \R)$, 
    $q \coloneq \min\{q_1, q_2, q_3\} < 2$. Additionally, $(\mu_i)_\Delta$ is bounded in 
    $\Wloc{-1}{r}(\open{0}{+\infty} \times \R, \R)$ for any $r \in \open{2}{+\infty}$. 
    Lemma \ref{lmm:CompCompTechnical} applies to ensure that for any $i \in \{1, 2\}$, 
    $(\mu_i)_\Delta$ belongs to a compact subset of $\Hloc{-1}(\open{0}{+\infty} \times \R, \R)$.
\end{proof}

\begin{corollary}
    \label{cor:CompComp}
    Assume that $V_{\min}, V_{\max} \in \W{2}{\infty}(\open{0}{1}, \R^+)$ satisfy 
    \eqref{eq:VelocitiesAssumption}-\eqref{eq:vacuum}.
    Fix $\rho_o \in \L{\infty}(\R, [\eps, 1])$ and $w_o \in \BV(\R, [0, 1])$. 
    Let $(\rho_\Delta, w_\Delta)_\Delta$ be the sequence generated by the scheme 
    described in Section \ref{sec:Main}.   
    Then, there exists a limit function 
    $\rho \in \L{\infty}(\open{0}{+\infty} \times \R, \R)$ 
    such that along a subsequence as $\Delta \to 0$, $(\rho_\Delta)_\Delta$ converges 
    in $\Lloc{p}(\open{0}{+\infty} \times \R, \R)$ for all 
    $p \in [1, +\infty \mathclose[$ and a.e. on 
    $\open{0}{+\infty} \times \R$ to $\rho$. 
\end{corollary}

\begin{proof}
    Combination of Lemma \ref{lmm:CompCompC} and Lemma \ref{lmm:CompComp}.
\end{proof}

\subsubsection{Convergence}

Notice that so far, the proof does not require $w_o$ to have any particular structure, meaning that for any $\BV$ function taking values in $[0, 1]$, the scheme will converge to a limit, and the limit is a weak solution of the system, as highlighted by Steps 2-3 below. It is both the piecewise constant structure of $w_o$, and the structural result Theorem \ref{th:BeyondPanov} \textit{(ii)} that allows to prove that the limit also has the piecewise constant structure, see Step 4.

\begin{proofof}{Theorem \ref{th:ConvergenceScheme}}
    Let $(\rho, w)$ be the limit functions from Proposition \ref{pp:Compactness1} and 
    Corollary \ref{cor:CompComp}. We prove that all the items of Definition \ref{def:gsom} hold.
    Let us split the proof into a few steps.

    \paragraph{Step 1.} 
    Let us introduce the piecewise constant function
    \[
        \overline{\rho}_\Delta 
        \coloneq \sum_{n=0}^{+\infty} \sum_{j \in \Z} \rho_j^n 
        \1_{[t^n, t^{n+1} \mathclose[ \times I_j}.        
    \]
    
    We claim that 
    $\norm{\rho_\Delta - \overline{\rho}_\Delta}_{\Lloc{2}(\R^+ \times \R)} \limit{\Delta}{0} 0$. 
    Indeed, for all 
    $(t,x) \in [t^n, t^{n+1}\mathclose[ \times \open{x_{j+1/2}}{x_{j+1}}$,
    \[
        \abs{\rho_\Delta(t, x) - \overline{\rho}_\Delta(t, x)}
        = \abs{\cU^n(t, x) - \rho_{j+1}^n} 
        \leq \abs{\rho_{j+1/2}^{n, +} - \rho_{j+1}^n},
    \]
    
    since $\cU^n$ is the solution of a Riemann problem with left state $\rho_j^n$ and 
    right state $\rho_{j+1}^n$ at $x = x_{j+1/2}$. Likewise, for all 
    $(t, x) \in [t^n, t^{n+1} \mathclose[ \times \open{x_j}{x_{j+1/2}}$,
    \[
        \abs{\rho_\Delta(t, x) - \overline{\rho}_\Delta(t, x)}
        = \abs{\cU^n(t, x) - \rho_j^n} 
        \leq \abs{\rho_{j+1/2}^{n, -} - \rho_j^n}.
    \]
    
    Therefore, by Lemma \ref{lmm:CompCompB},
    \[
        \begin{aligned}
            \norm{\rho_\Delta - \overline{\rho}_\Delta}_{\Lloc{2}(\R^+ \times \R)}^2 
            & = \sum_{n=0}^N \sum_{\abs{j} \leq J} 
            \int_{t^n}^{t^{n+1}} \int_{x_j}^{x_{j+1}} 
            \abs{\rho_\Delta(t, x) - \overline{\rho}_\Delta(t, x)}^2 \d{x} \d{t} \\
            & \leq \frac{1}{2} \sum_{n=0}^N \sum_{\abs{j} \leq J} 
            \left\{ \abs{\rho_{j+1/2}^{n, +} - \rho_{j+1}^n}^2 
            + \abs{\rho_{j+1/2}^{n, -} - \rho_j^n}^2 \right\} \Delta x \Delta t
            \leq \frac{c_2}{2} \Delta t,
        \end{aligned} 
    \]
    
    proving the claim.

    \paragraph{Step 2: $(\rho, f(\rho, w))$ is divergence free.}  
    
    Let $\phi \in \Cc{\infty}(\R^+ \times \R, \R)$ and fix $T > 0$ such that $\phi(t, x) = 0$ if $t > T$ and $x \in \R$. It will also be useful to fix $N, J \in \N \setminus \{0\}$ such that the support of $\phi$ is included in the subset $[0, t^N] \times [x_{-J}, x_J]$. Define 
    \[
        \forall n \in \N, \; \forall j \in \Z, \quad 
        \phi_j^n \coloneq \frac{1}{\Delta x} \int_{I_j} \phi(t^n, x) \d{x}.
    \]

    Multiply \eqref{eq:MF2} by $\phi_j^n \Delta x$ and take the double sum for 
    $n \in \N$ and $j \in \Z$. A summation by parts provides $A + B  = 0$ with
    \[
        A \coloneq \sum_{n=1}^{+\infty} \sum_{j \in \Z} 
        \rho_j^{n} (\phi_j^{n} - \phi_j^{n-1}) \Delta x 
        + \sum_{j \in \Z} \rho_j^{0} \phi_j^{0} \Delta x 
        \quad \text{and} \quad
        B \coloneq \sum_{n=0}^{+\infty} \sum_{j \in \Z} 
        F_{j+1/2}^n (\rho_j^n, \rho_{j+1}^n) (\phi_{j+1}^{n} - \phi_j^n) \Delta t.
    \]

    Clearly,
    \[
        A 
        = \int_{\Delta t}^{+\infty} \int_{\R} 
        \overline{\rho}_\Delta \p_t \phi \; \d{x}\d{t} 
        + \int_{\R} \rho_\Delta(0, x) \phi(0, x) \d{x} \limit{\Delta}{0} 
        \int_{0}^{+\infty} \int_{\R} \rho  \p_t \phi \; \d{x}\d{t} 
        + \int_{\R} \rho_o(x) \phi(0, x) \d{x}.
    \]

    On the other hand, using the definition of the interface flux,
    \[
        \begin{aligned}
            B & = \sum_{n=0}^{+\infty} \sum_{j \in \Z} \int_{t^{n}}^{t^{n+1}} 
            \int_{I_j} F_{j+1/2}^n (\rho_j^n, \rho_{j+1}^n) \p_x \phi \; \d{x} \d{t} \\
            & + \underbrace{\sum_{n=0}^{+\infty} \sum_{j \in \Z} 
            F_{j+1/2}^n (\rho_j^n, \rho_{j+1}^n) 
            \left\{ (\phi_{j+1}^{n} - \phi_j^n) \Delta t - 
            \int_{t^{n}}^{t^{n+1}} \int_{I_j} \p_x \phi \d{x} \d{t} \right\}}_{ \coloneq B_1} \\
            & = \int_{0}^{+\infty} \int_{\R}  
            f(\overline{\rho}_\Delta, w_\Delta) \p_x \phi \; \d{x} \d{t}
            + \underbrace{\sum_{n=0}^{+\infty} \sum_{j \in \Z} \int_{t^{n}}^{t^{n+1}} 
            \int_{I_j} (f(\rho_{j+1/2}^{n,-}, w_j^n) - f(\rho_j^n, w_j^n)) 
            \p_x \phi \; \d{x} \d{t}}_{ \coloneq B_2} + B_1.
        \end{aligned}
    \]

    Clearly,
    \[
        \abs{B_1} \leq \sup_{\omega, u \in [0, 1]} \abs{f(u, \omega)} \; 
        \left\{ T \; \norm{\p_{xx}^2 \phi}_{\L{\infty}(\R^+, \L{1}(\R))} \Delta x 
        + \norm{\p_{tx}^2 \phi}_{\L{1}(\R^+\times \R)} \Delta t \right\}, 
    \]

    while taking advantage of Lemma \ref{lmm:CompCompB}, 
    \[
        \begin{aligned}
            \abs{B_2} 
            & \leq L \sum_{n=0}^{+\infty} \sum_{j \in \Z} \int_{t^{n}}^{t^{n+1}} 
            \int_{I_j} \abs{\rho_{j+1/2}^{n,-} - \rho_j^n} \cdot \abs{\p_x \phi} \; \d{x} \d{t} \\
            & \leq L \left\{ \sum_{n=0}^{N} \sum_{\abs{j} \leq J} \abs{\rho_{j+1/2}^{n,-} - \rho_j^n}^2 \right\}^{1/2} \left\{ \sum_{n=0}^{N} \sum_{\abs{j} \leq J} 
            \left( \int_{t^{n}}^{t^{n+1}} \int_{I_j} \abs{\p_x \phi} \d{x} \d{t} \right)^2 \right\}^{1/2} \\
            & \leq L \sqrt{c_2} \; \norm{\p_x \phi}_{\L{2}(\R^+ \times \R)} \sqrt{\Delta t}.
        \end{aligned}
    \]

    These two bounds ensure that 
    \[
        B \limit{\Delta}{0} 
        \int_{0}^{+\infty} \int_{\R} f(\rho, w) \p_x \phi \; \d{x} \d{t},
    \]

    proving that $\rho$ is a weak solution of $\p_t \rho + \p_x f(\rho, w) = 0$, with initial datum $\rho_o$.

    \paragraph{Step 3: Item \ref{item:WeakFormulation} holds.}
    
    The key point is that \eqref{eq:MF2}-\eqref{eq:MFw} produces a conservative scheme 
    for $\rho w$. For the sake of clarity, we write below $F_{j+1/2}^n$ instead of 
    $F_{j+1/2}^n(\rho_j^n, \rho_{j+1}^n)$. Multiply \eqref{eq:MFw} by $\rho_j^{n+1}$ to get:
    \begin{equation}
        \begin{aligned}
            \label{eq:MF3}
            \rho_j^{n+1} w_j^{n+1} 
            & = (\rho_j^{n+1} - \lambda F_{j-1/2}^n) w_j^n
            + \lambda F_{j-1/2}^n w_{j-1}^n \\
            & = \rho_j^n w_j^n - \lambda \left( F_{j+1/2}^n w_j^n 
            - F_{j-1/2}^n w_{j-1}^n\right).
        \end{aligned}
    \end{equation}

    Keeping the notations of Step 2, multiply \eqref{eq:MF3} by $\phi_j^n \Delta x$ and take the double sum for $n \in \N$ and $j \in \Z$. A summation by parts provides $A + B = 0$ with
    \[
        \begin{aligned}
            A & \coloneq \sum_{n=1}^{+\infty} \sum_{j \in \Z} 
            \rho_j^{n} w_j^n  (\phi_j^{n} - \phi_j^{n-1}) \Delta x 
            + \sum_{j \in \Z} \rho_j^{0} w_j^0 \phi_j^{0} \Delta x \\
            B & \coloneq \sum_{n=0}^{+\infty} \sum_{j \in \Z} 
            F_{j+1/2}^n w_j^n (\phi_{j+1}^{n} - \phi_j^n) \Delta t.
        \end{aligned}
    \]

    Repeat the process of Step 2 to show that 
    \[
        A + B \limit{\Delta}{0} 
        \int_{0}^{+\infty} \int_{\R} \rho w \p_t \phi + f(\rho, w) w \p_x \phi \; \d{x} \d{t}
        + \int_{\R} \rho_o(x) w_o(x) \phi(0, x) \d{x},
    \]

    which concludes the proof of the claim. 

    Steps 2-3 ensure that $(\rho, w)$ is weak solution of \eqref{eq:gsom} with initial datum $(\rho_o, w_o)$. Assume now that $w_o$ is piecewise constant.

    \paragraph{Step 4: Item \ref{item:regularity} holds.} Since $w_o$ is piecewise constant, Steps 2-3 and Theorem \ref{th:BeyondPanov} ensure that there exist $M$ ordered curves $y_m \in \Lip(\open{t_m}{T_m}, \R)$ such that setting $\Gamma_m \coloneq \{(t, y_m(t)),\;t \in [t_m, T_m\mathclose[\}$, $w$ is constant in each connected component of $(\R^+ \times \R) \setminus \cup_{m=1}^M \Gamma_m$.

    \paragraph{Step 5: Discrete entropy inequalities.} Let $k \in [\eps, 1]$. Under the CFL condition \eqref{eq:CFL}, we derive the following discrete entropy inequalities, consequence of the monotonicity of the scheme:
    \begin{equation}
        \label{eq:DEI}
        \begin{aligned}
            & \left(\abs{\rho_j^{n+1} - k} - \abs{\rho_j^n - k} \right) \Delta x 
            + (\Phi_{j+1/2}^n - \Phi_{j-1/2}^n) \Delta t \\
            & \leq - \sgn(\rho_j^{n+1} - k) (F_{j+1/2}^n(k, k)- F_{j-1/2}^n(k, k)) \Delta t,
        \end{aligned}
    \end{equation}

    where 
    \[
        \Phi_{j+1/2}^n\coloneq F_{j+1/2}^n(\rho_j^n \vee k, \rho_{j+1}^n \vee k) 
        - F_{j+1/2}^n(\rho_j^n \wedge k, \rho_{j+1}^n \wedge k).
    \]

    Indeed, on the one hand, by convexity,
    \[
        \begin{aligned}
            \abs{\rho_j^{n+1} - H_j^n(k, k, k)} 
            & = \abs{\rho_j^{n+1} - k + \lambda (F_{j+1/2}^n(k, k) - 
            F_{j-1/2}^n(k, k))} \\
            & \geq \abs{\rho_j^{n+1} - k} + \lambda \sgn(\rho_j^{n+1} - k) 
            (F_{j+1/2}^n(k, k)- F_{j-1/2}^n(k, k)).
        \end{aligned}
    \]  

    On the other hand, by monotonicity of the scheme,
    \[
        \begin{aligned}
        \abs{\rho_j^{n+1} - H_j^n(k, k, k)}
        & \leq H_j^n(\rho_{j-1}^n \vee k, \rho_j^n \vee k, \rho_{j+1}^n \vee k) 
        - H_j^n(\rho_{j-1}^n \wedge k, \rho_j^n \wedge k, \rho_{j+1}^n \wedge k) \\
        & = \abs{\rho_j^n - k} - \lambda (\Phi_{j+1/2}^n - \Phi_{j-1/2}^n).
        \end{aligned}
    \] 

    Inequality \eqref{eq:DEI} follows by combining these two estimates.

    \paragraph{Step 6: Item \ref{item:EntropyInequalities} holds.} 
    As described in \cite[Section 4]{sylla_lwr_2022}, it suffices to prove the claim for $M=1$, which we do now, dropping the $m$ index in the process. Let $\varphi \in \Cc{\infty}(\R^+ \times \R, \R^+)$ and $k \in [\eps, 1]$. Fix $T, R > 0$ such that the support of $\varphi$ is included in $[0, T] \times [-R, R]$. Set for all $n \in \N$ and $j \in \Z$, $\varphi_j^n \coloneq \frac{1}{\Delta x} \int_{I_j} \varphi(t^n, x) \d{x}$. 

    Multiply \eqref{eq:DEI} by $\varphi_j^n \Delta x$ and take the double sum for $n \in \N$ and $j \in \Z$. A summation by parts provides $A + B + C \geq 0$ with
    \[
        \begin{aligned}
            A & \coloneq \sum_{n=1}^{+\infty} \sum_{j \in \Z} 
            \abs{\rho_j^{n}-k} (\varphi_j^{n} - \varphi_j^{n-1}) \Delta x 
            + \sum_{j \in \Z} \abs{\rho_j^{0}-k} \varphi_j^{0} \Delta x \\
            B & \coloneq \sum_{n=0}^{+\infty} \sum_{j \in \Z} 
            \Phi_{j+1/2}^n (\varphi_{j+1}^{n} - \varphi_j^n) \Delta t \\
            C & \coloneq \sum_{n=0}^{+\infty} \sum_{j \in \Z} 
            \sgn(\rho_j^{n+1} - k) (F_{j+1/2}^n(k, k) - 
            F_{j-1/2}^n(k, k)) \varphi_j^{n} \Delta t.
        \end{aligned}
    \]

    Clearly, $A \limit{\Delta}{0} \int_{0}^{+\infty} \int_{\R} \abs{\rho - k} \p_t \varphi \; \d{x}\d{t} + \int_{\R} \abs{\rho_o(x) - k} \varphi(0, x) \d{x}$. Consider now $B$ and write 
    \[
        \begin{aligned}
            \Phi_{j+1/2}^n
            & = \underbrace{F_{j+1/2}^n(\rho_j^n \vee k, \rho_{j+1}^n \vee k) 
            - F_{j+1/2}^n(\rho_j^n \vee k, \rho_j^n \vee k)}_{B_1} \\
            & + \underbrace{F_{j+1/2}^n(\rho_j^n \vee k, \rho_j^n \vee k) 
            - f(\rho_j^n \vee k, w_j^n)}_{B_{2}} + \; \Phi(\rho_j^n, w_j^n, k) + 
            \underbrace{f(\rho_j^n \wedge k, w_j^n) - 
            F_{j+1/2}^n(\rho_j^n \wedge k, \rho_j^n \wedge k)}_{B_3} \\
            & + \underbrace{F_{j+1/2}^n(\rho_j^n \wedge k, \rho_j^n \wedge k) 
            - F_{j+1/2}^n(\rho_j^n \wedge k, \rho_{j+1}^n \wedge k)}_{B_4}.
        \end{aligned}
    \]

    We see that
    \[
        \begin{aligned}
            \abs{\sum_{n=0}^{+\infty} \sum_{j \in \Z} (B_1 + B_4) 
            (\varphi_{j+1}^{n} - \varphi_j^n) \Delta t} & \leq  
            2 L \biggl( T \norm{\p_{xx}^2 \varphi}_{\L{\infty}(\R^+, \L{1}(\R))} 
            + \norm{\p_{tx}^2 \varphi}_{\L{1}(\R^+ \times \R)} \biggr) (\Delta x + \Delta t) \\
            & + 2 L \int_{0}^{+\infty} \int_{\R} \abs{\overline{\rho}_\Delta(t, x+\Delta x) - 
            \overline{\rho}_\Delta(t, x)} \cdot \abs{\p_x \varphi(t, x)} \d{x} \d{t}.
        \end{aligned}
    \]

    Since $(\overline{\rho}_\Delta)_\Delta$ converges a.e. on $\open{0}{+\infty} \times \R$, $(\overline{\rho}_\Delta \p_x \varphi)_\Delta$ is strongly compact in 
    $\L{1}(\open{0}{+\infty} \times \R, \R)$. As a consequence of the 
    Riesz-Fréchet-Kolmogorov compactness characterization, 
    \[
        \int_{0}^{+\infty} \int_{\R} \abs{\overline{\rho}_\Delta(t, x+\Delta x) - 
        \overline{\rho}_\Delta(t, x)} 
        \cdot \abs{\p_x \varphi(t, x)} \d{x} \d{t} \limit{\Delta}{0} 0.
    \]

    Using Lemma \ref{lmm:ConsistencyInterfaceFlux}, we can bound the contribution of $B_2$ and $B_3$ the following way:
    \[
        \begin{aligned}
            \abs{\sum_{n=0}^{+\infty} \sum_{j \in \Z} (B_2 + B_3) 
            (\varphi_{j+1}^{n} - \varphi_j^n) \Delta t} 
            & \leq 2 \mu \sum_{n=0}^{+\infty} \sum_{j \in \Z} \abs{w_{j+1}^n - w_j^n} \cdot 
            \abs{\varphi_{j+1}^{n} - \varphi_j^n} \Delta t \\
            & \leq 2 \mu \sum_{n=0}^{N} \TV(w_\Delta(t^n)) 
            \norm{\p_x \varphi}_{\L{\infty}(\R^+ \times \R)}\Delta x \Delta t 
            \limit{\Delta}{0} 0.
        \end{aligned}
    \]

    We now handle the last term in $B$. We have:
    \[
        \begin{aligned}
            & \sum_{n=0}^{+\infty} \sum_{j \in \Z} \Phi(\rho_j^n, w_j^n, k) 
            (\varphi_{j+1}^n - \varphi_j^n) \Delta t \\
            & = \Delta t \sum_{n=0}^{+\infty} \sum_{j \in \Z} \int_{I_j} \Phi(\rho_j^n, w_j^n, k) 
            \left( \frac{\varphi(t^n, x+ \Delta x) - \varphi(t^n, x)}{\Delta x} - 
            \p_x \varphi(t^n, x) \right) \d{x} \\
            & + \sum_{n=0}^{+\infty} \sum_{j \in \Z} \int_{t^n}^{t^{n+1}} \int_{I_j} \Phi(\rho_j^n, w_j^n, k) (\p_x \varphi(t^n, x) - \p_x \varphi(t, x)) \d{x} \d{t} \\
            & + \int_{0}^{+\infty} \int_{\R} \Phi(\overline{\rho}_\Delta, w_\Delta, k) \p_x \varphi(t, x) \d{x} \d{t} \\
            & \leq T \sup_{0 \leq p, \omega \leq 1} \abs{\Phi(p, \omega, k)} \cdot 
            \norm{\p_{xx}^2 \varphi}_{\L{\infty}(\R^+, \L{1}(\R))} \Delta x + 
            \sup_{0 \leq p, \omega \leq 1} \abs{\Phi(p, \omega, k)} \cdot 
            \norm{\p_{tx}^2 \varphi}_{\L{1}(\R^+ \times \R)} \Delta t \\
            & + \int_{0}^{+\infty} \int_{\R} \Phi(\overline{\rho}_\Delta, w_\Delta, k) \p_x \varphi(t, x) \d{x} \d{t} \\
            & \limit{\Delta}{0} \int_{0}^{+\infty} \int_{\R} \Phi(\rho, w, k) \p_x \varphi(t, x) \d{x} \d{t}.
        \end{aligned}
    \]

    Consider $C$. Introduce 
    \[
        F_\Delta(t, x) \coloneq \sum_{n=0}^{+\infty} \sum_{j \in \Z} F_{j+1/2}^n(k, k).
    \]

    Using Lemma \ref{lmm:ConsistencyInterfaceFlux} and Proposition \ref{pp:Compactness1}, we first establish that:
    \[
        \sup_{t \in \R^+} \sup_\Delta \TV(F_\Delta(t)) < +\infty, \quad
        \forall n \in \N, \; \sum_{j \in \Z} \abs{F_{j+1/2}^{n+1}(k, k) -F_{j+1/2}^n(k, k)}\Delta x \leq c (\Delta t + \Delta x).
    \]

    Therefore, we deduce that up to the extraction of a subsequence:
    \begin{itemize}
        \item $(F_\Delta)_\Delta$ converges in $\Czero(\R^+; \Lloc{1}(\R))$ to some function $F$, that can only be $(t, x) \mapsto f(k, w(t, x))$ ;
        \item $\TV(F_\Delta(t)) \limit{\Delta}{0} \TV(F(t))$, using the lower semi-continuity of the $\BV$ semi-norm and the special structure of $w \mapsto f(k, w)$ ;
        \item $(\p_x F_\Delta)_\Delta$ converges weakly* in the sense of measures to $(t, x) \mapsto \p_x(f(k, w(t, x)))$, which is equal to $(t, x) \mapsto (f(k, w_r) - f(k, w_l)) \delta_{x=y(t)}$.
    \end{itemize} 

    Then, write $C = C_1 + C_2$, where 
    \[
        \begin{aligned}
            C_1 & \coloneq 
            -\sum_{n=1}^{+\infty} \sum_{j \in \Z} 
            \sgn(\rho_j^n - k) (F_{j+1/2}^n(k, k) - 
            F_{j-1/2}^n(k, k)) (\varphi_j^{n-1} - \varphi_j^{n}) \Delta t \\
            C_2 & \coloneq 
            -\sum_{n=1}^{+\infty} \sum_{j \in \Z} \int_{I_j}
            \sgn(\rho_j^n - k) \p_x F_\Delta(t^n, x) \varphi_j^n \Delta t.
        \end{aligned}
    \]

    On the one hand, it holds:
    \[
        \begin{aligned}
            \abs{C_1} 
            \leq T \; \sup_{t \in \R^+} \sup_\Delta \TV(F_\Delta(t))
            \cdot \norm{\p_t \varphi}_{\L{\infty}(\R^+ \times \R)} \Delta t.
        \end{aligned}
    \]

    On the other hand, we have:
    \[
        \lim_{\Delta \to 0} C_2 = -\int_{0}^{+\infty} \int_{\R} \sgn(\rho -k) \p_x F(t, x) \varphi(t, x) \d{x} \d{t} 
        \leq \int_{0}^{+\infty} \abs{f(k, w_r) - f(k, w_l)}\varphi(t, y(t)) \; \d{t},
    \]

    This ensures that $\rho$ satisfies the entropy inequalities for $k \in [\eps, 1]$. Now, if $k \in [0, \eps]$, since $\rho$ is a weak solution of the conservation law (Step 2) and takes its values in $[\eps, 1]$, we have 
    \[
            \begin{aligned}
                & \int_0^{+\infty} \int_\R \bigl( \abs{\rho - k} \p_t \varphi 
                + \Phi(\rho, w, k) \p_x \varphi \bigr) \d{x} \d{t} 
                + \int_\R \abs{\rho_o(x) - k} \varphi(0, x) \d{x} \\
                & = \int_0^{+\infty} \int_\R \bigl( (k-\rho) \p_t \varphi 
                + (f(k, w) - f(\rho, w)) \p_x \varphi \bigr) \d{x} \d{t} 
                + \int_\R (k-\rho_o(x)) \varphi(0, x) \d{x} \\ 
                & = \int_0^{+\infty} \left\{ \int_{-\infty}^{y(t)} f(k, w_l) \p_x \varphi \d{x} + \int_{y(t)}^{+\infty} f(k, w_r) \p_x \varphi \d{x}\right\} \d{t}  \\
                & = \int_0^{+\infty} (f(k, w_l) - f(k, w_r)) \varphi(t, y(t)) \d{t}
                \geq - \int_0^{+\infty} \abs{f(k, w_l) - f(k, w_r)} \varphi(t, y(t)) \d{t},
            \end{aligned}
    \]

    therefore the entropy inequalities are also valid for $k \in [0, \eps]$.
    
    The proof is concluded.
\end{proofof}

\newpage

{\small
  \bibliography{references}
  \bibliographystyle{abbrv}
}

\end{document}